\documentclass[11pt]{article}

\usepackage[a4paper, margin=1in]{geometry}
\usepackage[numbers,sort&compress]{natbib}
\usepackage{graphicx}
\usepackage{amsmath}
\usepackage{amsfonts}
\usepackage{amssymb}
\usepackage{amsthm}
\usepackage{textgreek}
\usepackage{xcolor}
\usepackage[version=4,arrows=pgf-filled,textfontname=sffamily,mathfontname=mathsf]{mhchem}
\usepackage{microtype}
\usepackage{xurl}
\usepackage{authblk}
\usepackage[colorlinks=true,
            linkcolor=blue,
            citecolor=blue,
            urlcolor=blue,
            pdfborder={0 0 0}]{hyperref}

\theoremstyle{plain}
\newtheorem{theorem}{Theorem}[section]
\newtheorem{proposition}[theorem]{Proposition}

\newtheorem{lemma}[theorem]{Lemma}
\newtheorem{definition}[theorem]{Definition}
\theoremstyle{remark}
\newtheorem{remark}[theorem]{Remark}

\numberwithin{equation}{section}
\numberwithin{figure}{section}
\numberwithin{table}{section}

\newenvironment{keywords}{\par\medskip\noindent\textbf{Keywords:}}{\par}
\newenvironment{MSCcodes}{\par\smallskip\noindent\textbf{MSC codes:}}{\par}

\renewcommand{\vec}[1]{\boldsymbol{#1}}

\title{Structural functional identifiability and model discovery in differential equation
    models}
\author[1]{Torkel E. Loman\thanks{Corresponding author: \href{mailto:torkel.loman@maths.ox.ac.uk}{torkel.loman@maths.ox.ac.uk}}}
\author[2]{Alexander P. Browning}
\author[1]{Ruth E. Baker}
\affil[1]{Mathematical Institute, University of Oxford, Oxford, United Kingdom}
\affil[2]{School of Mathematics and Statistics, University of Melbourne, Melbourne, Australia}
\date{}

\begin{document}

\maketitle

\begin{abstract}
Differential equation models are widely used to describe, interpret, and predict dynamical phenomena across science and engineering. In practice, however, the governing dynamics are rarely fully known and must be inferred from observational data. Traditionally, inverse problems in differential equation modelling have focused on estimating unknown parameter values. In this setting, structural identifiability determines whether parameter values can, in principle, be uniquely recovered from ideal observations and is, therefore, a prerequisite for meaningful inference. More recently, the integration of machine learning with mechanistic modelling has enabled the discovery of unknown equations, functions, and constitutive relationships, substantially expanding the space of admissible models. This raises a fundamental question: under what conditions can unknown functional components be uniquely recovered from data? In this paper, we generalise the classical notion of structural parameter identifiability to functional identifiability. We first identify broad classes of models for which unique functional recovery is impossible. We then show how functional identifiability can be assessed for differential equation models using differential algebra-based techniques which are well-established as a means of assessing structural identifiability for ordinary differential equation models. Our framework reveals new phenomena that arise in the transition from parametric to function inference and have no analogue in the classical setting. Finally, we characterise functional identifiability in several common model classes. Taken together, our results demonstrate that functional identifiability provides a theoretical foundation for modern inverse problems in differential equation modelling, particularly those that use machine learning representations of unknown system components.
\end{abstract}



\section{Introduction}

Mathematical models are a widespread tool for understanding and predicting the behaviours of natural and human-made systems. Traditionally, such models are mechanistic: their governing equations are derived from physical, biological, or other domain-specific principles. In practice, however, these principles rarely determine a model completely. Mechanistic models, therefore, typically contain unknown quantities that must be inferred from observational data before the models can be used for prediction or decision-making. Historically, these unknowns have been finite-dimensional parameters, leading to a rich literature on parameter estimation and identifiability~\cite{bellman_structural_1970,raue_structural_2009,Chis_struct_ident_2011,villaverde_protocol_2022,simpson_parameter_2026}. Recent advances in machine learning have substantially expanded the range of unknown model components that can be inferred from data. At one extreme, neural ordinary differential equations replace the entire right-hand side of a model with a neural network, discarding mechanistic knowledge altogether~\cite{chen_neural_2019,kidger_neural_2020}. Between this and fully mechanistic modelling lie hybrid approaches, which retain the known parts of a model and learn only the unknown ones. Universal differential equations embed neural networks or other function approximators directly in the model equations~\cite{rackauckas_universal_2021,loman_functional_2025}, while physics-informed neural networks and their biologically-informed variants use the governing equations as constraints when fitting models to data~\cite{raissi_physics-informed_2019,karniadakis_physics-informed_2021,lagergren_biologically-informed_2020,yazdani_systems_2020}. The learnt functions are frequently then converted into equational expressions using equation discovery methods such as sparse regression or symbolic regression~\cite{brunton_discovering_2016,schmidt_distilling_2009}. Hybrid models of this kind have found applications across the physical, biological, and engineering sciences~\cite{alber_integrating_2019,noordijk_rise_2024,willard_integrating_2022,reichstein_deep_2019}.

In this work, we focus on \emph{function inference}, where the structure of a differential equation model is assumed known up to one or more unknown functions that must be learnt from data (Figure~\ref{fig:ude_workflow}). These functions may represent unknown rate laws, interaction terms, transfer processes, feedback relationships, potential fields, or other system components with forms that are not known \textit{a priori}. Unlike scalar parameters, functions are infinite-dimensional objects, so inferring them substantially expands the space of admissible fitted models. In practice, these functions can be learnt from data by first representing them using flexible approximators, such as neural networks or Gaussian processes. Many machine learning-based approaches for inferring model components, including universal differential equations, many physics-informed neural network applications, and most equation discovery methods, can be considered approaches for function inference. This shift from parameter inference to function inference raises a natural question: under what conditions can an unknown function be uniquely recovered from data?

\begin{figure}
    \centering
    \includegraphics[width=0.95\linewidth]{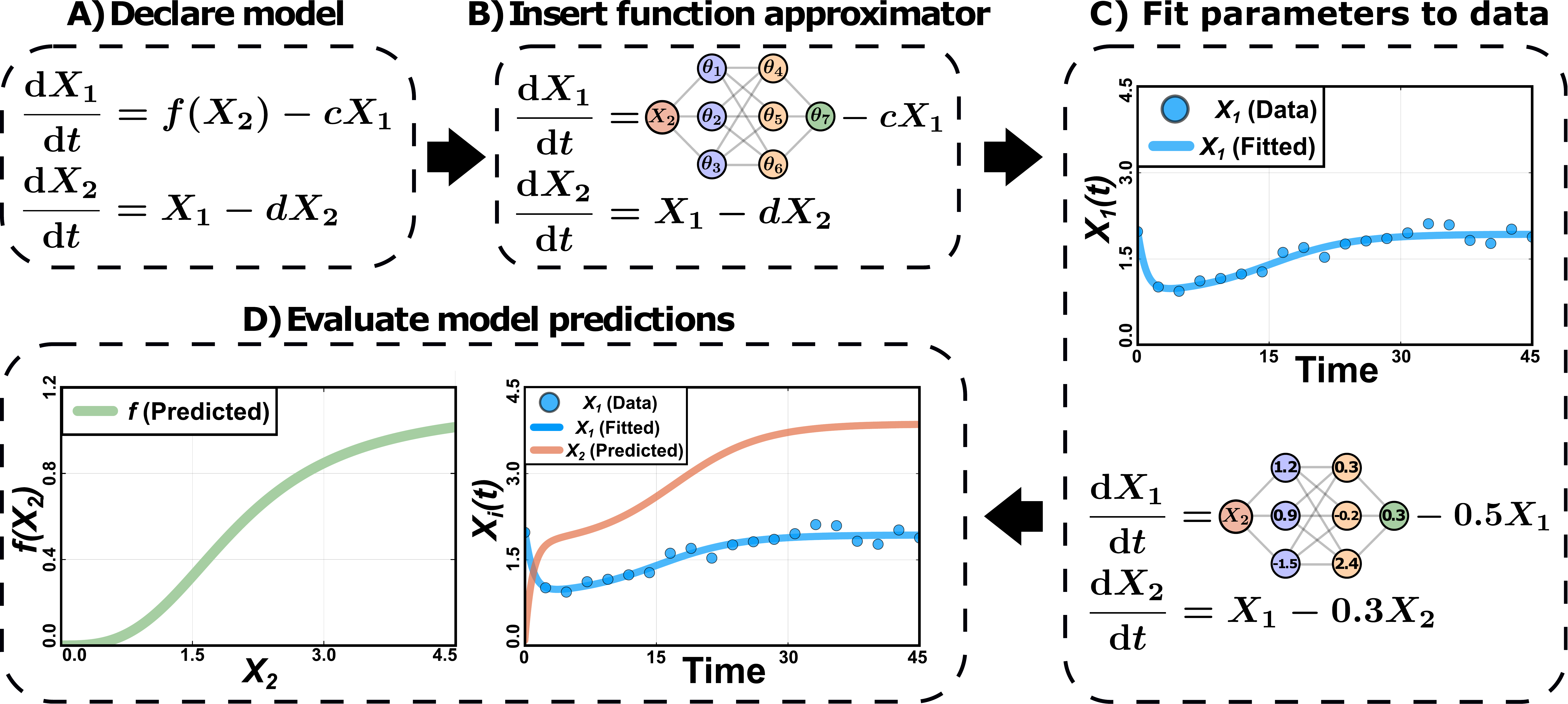}
    \caption{\textbf{Typical function learning workflow.} Example workflow where a functional form is learnt using a so-called universal differential equation model combined with data. (A) A differential equation model is formulated, including unknown parameters (here $c$ and $d$) and functions (here $f$). (B) The unknown function is replaced with a universal function approximator (here a neural network) which represents an arbitrary function through a finite set of parameters. (C) The model's scalar parameters (both mechanistic parameters and those representing the unknown function) are fitted to data using a standard parameter fitting workflow. (D) The fitted model can be evaluated to, for example, predict the time evolution of unobserved variables (here $X_2(t)$) or the fitted function.}
    \label{fig:ude_workflow}
\end{figure}

In the parameter estimation literature, \emph{identifiability} concerns whether model parameters can be uniquely determined from observational data~\cite{bellman_structural_1970,cobelli_parameter_1980,hines_determination_2014,villaverde_protocol_2022}. A distinction is commonly drawn between \emph{structural} and \emph{practical} identifiability. Methods for structural identifiability analysis ask whether distinct parameter values produce identical observable dynamics and thus remain indistinguishable even with unlimited noiseless data~\cite{bellman_structural_1970,Chis_struct_ident_2011,meshkat_algorithm_2009}. Practical identifiability, by contrast, concerns whether parameters can be reliably estimated from the finite and noisy data available in practice~\cite{raue_structural_2009,simpson_parameter_2026,wieland_structural_2021}. Structural identifiability provides the theoretical foundation for parameter inference, and hence it should be assessed before both practical identifiability analysis and parameter inference.

The shift from parameter inference to function inference raises a natural analogue of this classical identifiability problem. Here we develop a theory of \emph{structural functional identifiability} for ordinary differential equation (ODE) models containing unknown functions. Our contributions are fourfold. First, we identify broad classes of models in which functional non-identifiability arises naturally and show that, in these cases, hybrid models are theoretically equivalent to neural differential equations with entirely unknown right-hand sides~\cite{chen_neural_2019,kidger_neural_2020}. Second, we extend differential algebra-based methods to assess functional identifiability. Third, we demonstrate that structural functional identifiability exhibits phenomena with no direct analogue in the parametric setting, including entanglement between unknown functions and parameters, entanglement between multiple unknown functions, and potential non-identifiability even when the only unknown is a single function. Finally, we illustrate these methods on commonly used model types, including a chemical reaction network model and the Lotka--Volterra model.


\section{Structural functional identifiability: definition and scope of this work} \label{sec:functional_identifiability_definition}

Structural functional identifiability asks whether an unknown function can, in theory, be determined from perfect observations of a model's output. Consider an ODE model
\begin{align*}
    \dot{\vec{X}}(t)&= \vec{F}(\vec{X};\, \vec{f},\, \vec{p}), \\
    \vec{Y}(t)&=\vec{H}(\vec{X}(t)),
\end{align*}
where $\vec{X}\in\mathbb{R}^n$ is the state vector, $\vec{Y}\in\mathbb{R}^l$ is the observed output, $\vec{F}$ and $\vec{H}$ are known functions, $\vec{p}\in\mathbb{R}^m$ is a vector of unknown scalar parameters, and $\vec{f}=(f_1,\ldots,f_k)$ is a collection of unknown functions with
$f_j:\mathbb{R}^{n_j}\to\mathbb{R}$ and $n_j \leq n$. Two choices of the unknowns, $(\vec{f}_1,\vec{p}_1)$ and $(\vec{f}_2,\vec{p}_2)$, are said to give the same observed outputs if they produce the same output trajectory $\vec{Y}(t)$. Throughout this work, we will assume that each observable in $\vec{Y}$ is identical to a state variable in $\vec{X}$. More complicated observable formulae, like $Y_1(t) = \sum X_i(t)$, are not considered here (although methodologies for the parametric case could be translated to the function setting~\cite{hong_global_2020}).

\begin{definition} \label{def:functional_identifiability}
Fix a generic choice of the unknowns $(\vec{f}_1,\vec{p}_1)$. An unknown function $f_j$, $j=1,\dots,k$, is:
\begin{itemize}
    \item \emph{Globally structurally identifiable} if every choice
    of unknowns $(\vec{f}_2,\vec{p}_2)$ that gives the same observed outputs as
    $(\vec{f}_1,\vec{p}_1)$ has $f_{1,j}=f_{2,j}$.
    \item \emph{Locally structurally identifiable} if there exists
    $\varepsilon>0$ such that every choice of unknowns $(\vec{f}_2,\vec{p}_2)$ that gives the
    same observed outputs as $(\vec{f}_1,\vec{p}_1)$ and satisfies
    $\sup_{\vec{u}\in\mathbb{R}^{n_j}}|f_{1,j}(\vec{u})-f_{2,j}(\vec{u})|<\varepsilon$
    must have $f_{1,j}=f_{2,j}$\footnote{This is equivalent to there being a \emph{countable} number of functions $f_j$ giving the same observed dynamics. For systems with rational right-hand sides, a \emph{finite} number of functions suffice.}.
    \item \emph{Structurally non-identifiable} otherwise.
\end{itemize}
\end{definition}

\noindent
Parameter identifiability is defined analogously, but with $p_{1,j}=p_{2,j}$. Throughout this work we will mainly consider global structural identifiability, for simplicity termed structural identifiability unless otherwise mentioned. Local structural identifiability will be considered in more detail in Section~\ref{sec:local_global}.

Determining whether a model's components are structurally identifiable has both practical and theoretical applications. Primarily, it is used to determine what model components can be recovered from data. However, it can also guide experimental design: by investigating how structural identifiability depends on the observed quantities, one can better determine what system components to measure to maximise identifiability.

This paper is laid out as follows. Section~\ref{sec:general_nonident} identifies several broad classes of models for which structural functional non-identifiability holds generally. The aim is not to provide a comprehensive catalogue of non-identifiable models, but to highlight common cases where functional recovery is an ill-posed problem. Section~\ref{sec:diffalg} then illustrates, through carefully chosen examples, how differential algebra techniques can be used to determine functional identifiability in ODE models. This approach can handle models with diverse algebraic structures, and is therefore applicable beyond the special forms considered in Section~\ref{sec:general_nonident}. Section~\ref{sec:nonidentifiability_types} uses the same approach to distinguish different forms of functional non-identifiability that can be encountered, considering both global and local structural identifiability. Finally, Section~\ref{sec:applications} applies the differential algebra approach to two common model categories, including cases where unknown functions appear multiple times in the same model or depend on multiple variables.


\section{Structural non-identifiability for general classes of models} \label{sec:general_nonident}

In this section, we identify broad classes of ODE models for which unique functional recovery is structurally impossible. For each class, we construct an explicit family of functions that produces identical observable dynamics and is, therefore, functionally indistinguishable. Throughout, we assume that all state variables are observable, so that $\vec{Y}(t)=\vec{X}(t)$. Consequently, the same non-identifiability results also apply when only a subset of variables is observed. Proofs of each result are provided in Appendix~\ref{appendix:proofs_section2}. We note that the differential equations considered here are by no means exhaustive of all settings in which unknown functions can occur. For example, here unknown parameters and functions occur in separate terms, which is not required. In the next section, we will show how to determine structural identifiability more generally.


\subsection{Scalar ODE models} \label{sec:nonident_1var_UDEs}

Consider a scalar ODE model containing both known and unknown dynamics:
\begin{equation} \label{eq:nonident_1var_UDEs}
    \dot{X} = f(X) + g(X;\vec{p}),
\end{equation}
where $f:\mathbb{R}\to\mathbb{R}$ is an unknown function and $g:\mathbb{R}\to\mathbb{R}$ is a known functional form depending on a vector $\vec{p}$ with $\operatorname{dim}(\vec{p})>0$ (i.e. at least one unknown parameter is present).

\begin{proposition} \label{prop:1var_nonident}
The unknown function $f$ is structurally non-identifiable. Specifically, given any admissible pair $(f_1, \vec{p}_1)$, any pair $(f_2, \vec{p}_2)$ of the form
\begin{align}    
    \vec{p}_2 &\in \mathbb{R}^{\operatorname{dim}(\vec{p})}, \\
    f_2(X) &= f_1(X) + g(X;\vec{p}_1) - g(X;\vec{p}_2),
\end{align}
yields identical dynamics to $(f_1, \vec{p}_1)$. 
\end{proposition}

The non-identifiability arises because the unknown function $f$ can absorb arbitrary changes in the parametric term $g(X;\vec{p})$. Observable data therefore constrain only the combined dynamics $f(X)+g(X;\vec{p})$, rather than the individual contributions of the known and unknown components. Consequently, the mechanistic component $g$ and the learnt component $f$ cannot be disentangled from observations alone: any pair $(f_2,\vec{p}_2)$ belonging to the family described by Proposition~\ref{prop:1var_nonident} yields identical dynamics, making simultaneous recovery of the parameter values and the unknown function structurally impossible. Example models satisfying the conditions of Proposition~\ref{prop:1var_nonident} are provided in Section~\ref{si:examples_1var}, with Figure~\ref{fig:single_var_nonident} providing a simple example.

\begin{figure}
    \centering
    \includegraphics[width=0.90\linewidth]{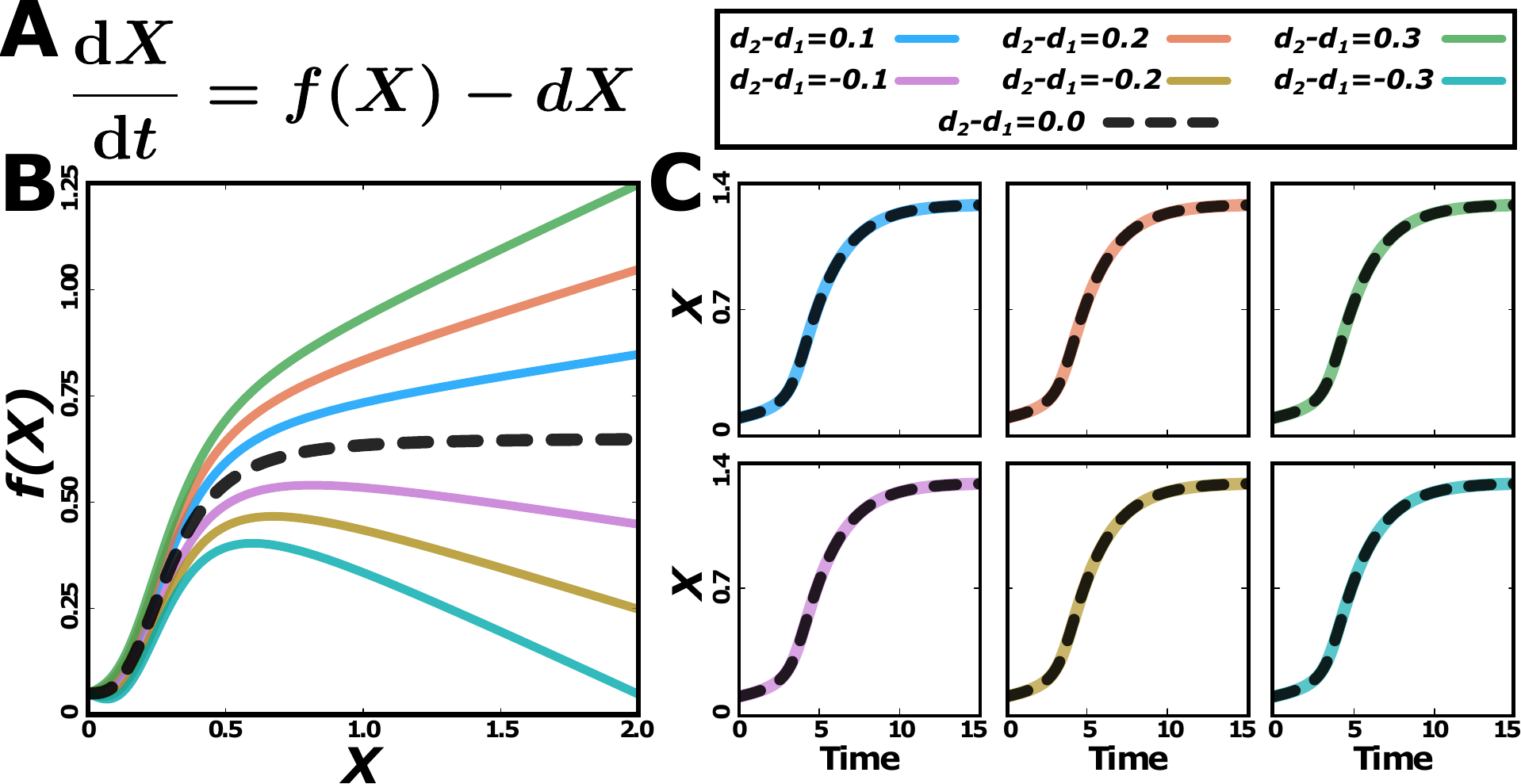}
    \caption{\textbf{A simple self-activation loop model is inherently structurally non-identifiable.} (A) A self-activation loop where a single component, $X$, activates its own production according to an unknown function $f(X)$ and decays at a linear (unknown) rate $d$. According to Proposition~\ref{prop:1var_nonident}, this model is structurally non-identifiable. The reason is that, given any base pair $(f_1(X),d_1)$, any alternative pair $d_2 \in \mathbb{R}$, $f_2(X) = f_1(X) + (d_2 - d_1)X$ will yield identical dynamics. (B) Using a sigmoidal base function $f_1(X)$ (black dashed line), we plot potential functions $f_2(X)$ for different choices of $d_2$ (coloured lines). (C) For each pair $(f_2(X),d_2)$ in B, we simulate the ODE in A (coloured lines) and compare these to the simulation for $(f_1(X),d_1)$. These simulations confirm that the different model alternatives all yield identical dynamics. While we here only consider a single initial condition, Proposition~\ref{prop:1var_nonident} tells us the pairs will yield matching simulations for any initial condition.}
    \label{fig:single_var_nonident}
\end{figure}


\subsection{Partially augmented systems of ODEs} \label{sec:nonident_partial_system_UDEs}

The previous example can be extended to systems of ODEs by adding additional variables not directly involved in the non-identifiability relation. As a concrete example, consider an oscillator with position $X$ and velocity $V$, governed by the ODE model
\begin{align*}
    \dot{X} &= V, \\
    \dot{V} &= f(X) - kX - cV,
\end{align*}
where $f:\mathbb{R}\to\mathbb{R}$ is an unknown restoring-force contribution and $k>0$ and $c>0$ are unknown stiffness and damping parameters. The acceleration depends on the sum $f(X)+kX$, so any change in $k$ can be offset by tuning a linear term in $f(X)$.

We can generalise this model structure by considering the following system in which one equation contains an unknown function:
\begin{align}
    \dot{X}_1 &= f(\vec{u}) + g_1(\vec{X};\vec{p}_1),  \label{eq:nonident_partial_system_UDEs} \\
    \dot{X}_i &= g_i(\vec{X};\vec{p}_i), \qquad i = 2,\ldots,n, \notag
\end{align}
where $\vec{u} \subseteq \vec{X}$, $f$ is unknown, the $g_i$ are known functional forms depending on unknown parameter sets $\vec{p}_i$, and $\operatorname{dim}(\vec{p}_1) > 0$. Suppose further that $g_1$ can be decomposed as
\begin{equation*}
    g_1(\vec{X};\vec{p}_1) = r(\vec{v};\vec{a}) + s(\vec{X};\vec{b}),
\end{equation*}
where $\vec{v} \subseteq \vec{u}$, $\vec{a} \cap \vec{b} = \emptyset$, $\vec{a} \cap \vec{p}_i = \emptyset$ for all $i > 1$, and $\operatorname{dim}(\vec{a}) > 0$. In practice, this means that the known part of the first equation contains a parameter-dependent term acting on variables that are arguments of the unknown function $f$. Furthermore, the parameters of this term cannot occur in any other equation.

\begin{proposition} \label{prop:multivar_nonident}
For the model in Equation~\eqref{eq:nonident_partial_system_UDEs}, $f$ is structurally non-identifiable. Specifically, given any admissible pair $(f_1, \vec{p}_{1,1}) \equiv (f_1, \{\vec{a}_{1}, \vec{b}\})$, any pair $(f_2, \vec{p}_{1,2}) \equiv (f_2, \{\vec{a}_{2}, \vec{b}\})$ of the form
\begin{align}
    \vec{a}_2 &\in \mathbb{R}^{\operatorname{dim}(\vec{a})}, \notag \\
    f_2(\vec{u}) &= f_1(\vec{u}) + r(\vec{v};\vec{a}_1) - r(\vec{v};\vec{a}_2),
\end{align}
yields identical dynamics to $(f_1, \vec{p}_{1,1})$.
\end{proposition}

As in Proposition~\ref{prop:1var_nonident}, the non-identifiability arises because changes in the known component of the equation for $X_1$ can be absorbed into the unknown function without altering the observable dynamics. Example models satisfying the conditions of Proposition~\ref{prop:multivar_nonident} are provided in Section~\ref{si:examples_multivar}. We note that this result is only guaranteed to hold when $f$ appears only in the equation for $X_1$; occurrences of $f$ in the equations for $X_2,\ldots,X_n$ can introduce additional constraints and may restore identifiability, as shown in Appendix~\ref{appendix:repeated_function_counterexample}. 


\subsection{Fully augmented systems of ODEs} \label{sec:nonident_complete_UDEs}

Next, we consider the case of a system of fully augmented ODEs where every equation contains an additive unknown function of the full state vector:
\begin{equation}
    \label{eq:complete_UDE}
    \dot{X}_i = f_i(\vec{X}) + g_i(\vec{X};\vec{p}_i), \qquad i = 1,\ldots,n,
\end{equation}
where $\operatorname{dim}(\vec{p}_i)>0$, that is, every equation contains at least one unknown parameter. Note that the different $\vec{p}_i$ need not be disjoint, and may overlap partially or completely.

\begin{proposition} \label{prop:complete_UDE_nonident}
All unknown functions $f_i$ are structurally non-identifiable. Specifically, given any admissible pair $(f_{i,1}, \vec{p}_{i,1})$, any pair of the form
\begin{align}
    \vec{p}_{i,2} &\in \mathbb{R}^{\operatorname{dim}(\vec{p}_i)}, \notag \\
    f_{i,2}(\vec{X}) &= f_{i,1}(\vec{X}) + g_i(\vec{X};\vec{p}_{i,1}) - g_i(\vec{X};\vec{p}_{i,2}),
\end{align}
with the choices of $\vec{p}_{i,2}$ required to be consistent on any overlaps between the $\vec{p}_i$, yields dynamics identical to $(f_{i,1}, \vec{p}_{i,1})$ for all $i$.
\end{proposition}

For example, the dynamics of a microbial biomass $X$ growing on a nutrient $S$ could be written
\begin{align*}
    \dot{X} &= f_X(X,S) - dX - cX^2, \\
    \dot{S} &= -f_S(X,S) - \lambda S,
\end{align*}
where $f_X:\mathbb{R}^2\to\mathbb{R}$ and $f_S:\mathbb{R}^2\to\mathbb{R}$ are unknown growth and nutrient-consumption laws, while $d$, $c$, and $\lambda$ are unknown decay parameters. Even with both $X$ and $S$ observed, $d$ and $c$ can be tuned arbitrarily with $f_X$ being adjusted accordingly to preserve the dynamics of $X$, and similarly for $f_S$ and $\lambda$. Further examples are provided in Section~\ref{si:examples_complete}. 


\subsection{Neural ODEs} \label{sec:nonident_neural_odes}

Finally, we consider a system of ODEs where each right-hand side is an unknown function of the full state space, i.e. no mechanistic knowledge is provided. These models are also known as \emph{neural ODEs}~\cite{chen_neural_2019,kidger_neural_2020}, and can be written as 
\begin{equation} \label{eq:nonident_neural_odes}
    \dot{X}_i = f_i(\vec{X}),
\end{equation}
where each $f_i$ is an unknown function.

\begin{proposition} \label{prop:nonident_neural_odes}
Neural ODEs are structurally identifiable if and only if the full state vector, $\vec{X}$, is observable.
\end{proposition}

This result shows that, for fully observed ODE systems (i.e. where each state can be observed directly), the right-hand sides are always structurally identifiable. This is particularly notable, as a neural ODE is the baseline, fully data-driven, case where no mechanistic knowledge is provided. In several fully observed non-identifiable models (including those described in Sections~\ref{sec:nonident_1var_UDEs} and \ref{sec:nonident_complete_UDEs}), the identifiable expressions are exactly the set of ODE right-hand side expressions. For these models, proposing a hybrid formulation therefore contributes no mechanistic knowledge at all: fitting the ODE model incorporating unknown functions to data is equivalent to fitting a fully data-driven neural ODE. This highlights a practical role for structural functional identifiability: determining when mechanistic assumptions genuinely provide identifiable information and when they do not.

The results in this section identify broad settings where functional recovery fails by construction. Next, we will describe how the differential algebra approach provides a practical framework for determining whether generic, partially observed, models are structurally identifiable.


\section{Determining structural identifiability via the differential algebra approach} \label{sec:diffalg}


Differential algebra provides a means to remove unobserved state variables from an ODE model through repeated differentiation and substitution. The result is an \emph{input--output equation}: a differential equation involving only observed variables, their derivatives, and the unknown functions and parameters (in some cases, a model can correspond to a set of input--output equations). If two choices of the unknowns (parameters and functions) generate the same input--output equation, then they cannot be distinguished given the observed dynamics (implying structural non-identifiability). In practice, this often means that one can consider each term of the input--output equation in isolation, from which identifiable quantities can be recovered~\cite{ljung_global_1994,Ovchinnikov2023}.

Consider, for example, the following system of ODEs,
\begin{equation*}
    \dot X_1=-aX_1+bX_2,\qquad
    \dot X_2=cX_1-dX_2,
\end{equation*}
with only $X_1$ observed and where $a$, $b$, $c$, and $d$ are unknown parameters. Eliminating the unobserved variable $X_2$ by differentiating and substituting gives the input--output equation
\begin{equation*}
    \ddot X_1+(a+d)\dot X_1+(ad-bc)X_1=0.
\end{equation*}
Structural identifiability considers arbitrary amounts of noiseless data. As is standard in the differential algebra approach, we therefore take the coefficients of the input--output equation, once normalised so that its highest-order derivative has coefficient one, to be structurally identifiable~\cite{ljung_global_1994,pia_saccomani_parameter_2003,hong_global_2020,ovchinnikov_computing_2021}. In this case, the two expressions $a+d$ and $ad-bc$ are identifiable. However, as multiple choices of $a$, $b$, $c$, and $d$ preserve these coefficients, the individual parameters are not structurally identifiable from observations of $X_1$ alone. If we also have that $d$ is known, then this gives that $a$ is identifiable (as well as the product $bc$, but not the individual parameters $b$ and $c$).


\subsection{Single-parameter mutual activation loop model} \label{sec:diffalg_one_parameter}

To demonstrate how the differential algebra approach can be applied to determine functional identifiability, we first consider a mutual activation loop model. Here, two variables mutually activate each other and decay at the same constant rate:
\begin{align}
    \dot{X}_1 &= f(X_2) - dX_1, \label{eq:diffalg_model1_dXdt} \\
    \dot{X}_2 &= X_1 - dX_2,   \label{eq:diffalg_model1_dYdt}
\end{align}
with $f:\mathbb{R}\to\mathbb{R}$ an unknown function and $d$ an unknown parameter.

\begin{proposition} \label{prop:diffalg_one_parameter_identifiability}
In the single-parameter mutual activation model, Equations~\eqref{eq:diffalg_model1_dXdt}--\eqref{eq:diffalg_model1_dYdt}, both $f$ and $d$ are structurally identifiable in either of the following cases:
\renewcommand{\labelenumi}{(\roman{enumi})}
\begin{enumerate}
    \item When only $X_2$ is observable.
    \item When only $X_1$ is observable, provided $f$ is locally invertible and $d \neq 0$.
\end{enumerate}
\end{proposition}

\begin{proof}[Proof of part (i)]
To eliminate $X_1$, we first isolate it in Equation~\eqref{eq:diffalg_model1_dYdt} and differentiate:
\begin{equation*}
    X_1= \dot{X}_2 + dX_2 \quad \Longrightarrow \quad \dot{X}_1 = \ddot{X}_2 + d\dot{X}_2.
\end{equation*}
Substituting into Equation~\eqref{eq:diffalg_model1_dXdt} gives the following input--output equation
\begin{equation}
    \label{eq:diffalg_model1_Yobs_dY2dt2}
    \ddot{X}_2 + 2d\dot{X}_2 + d^2 X_2 - f(X_2) = 0.
\end{equation}
Now suppose two pairs $(f_1, d_1)$ and $(f_2, d_2)$ yield identical $X_2$ dynamics. Inserting both into Equation~\eqref{eq:diffalg_model1_Yobs_dY2dt2}, isolating $\ddot{X}_2$, and then equating gives
\begin{equation}
    \label{eq:diffalg_model1_Yobs_equiveq}
    -2d_1\dot{X}_2 - d_1^2 X_2 + f_1(X_2) = \ddot{X}_2 = -2d_2\dot{X}_2 - d_2^2 X_2 + f_2(X_2).
\end{equation}
Here, under the assumption of an arbitrary amount of noiseless data, we can consider the $\dot{X}_2$-dependent and $\dot{X}_2$-independent parts of Equation~\eqref{eq:diffalg_model1_Yobs_equiveq} separately, allowing us to split it into two independent equalities:
\begin{align}
    -2d_1\dot{X}_2       &= -2d_2\dot{X}_2,       \label{eq:diffalg_model1_Yobs_split1} \\
    -d_1^2 X_2 + f_1(X_2) &= -d_2^2 X_2 + f_2(X_2).  \label{eq:diffalg_model1_Yobs_split2}
\end{align}
Equation~\eqref{eq:diffalg_model1_Yobs_split1} immediately gives $d_1 = d_2 \equiv d$. Substituting into Equation~\eqref{eq:diffalg_model1_Yobs_split2} then gives $f_1(X_2) = f_2(X_2)$. This demonstrates that two different pairs $(f, d)$ cannot yield identical dynamics with respect to $X_2$, showing that both $d$ and $f$ are structurally identifiable.
\end{proof}

The proof of part~(ii), which is given in Section~\ref{si:proof_case2}, is analogous to the proof above, but with $X_2$ eliminated rather than $X_1$. Since $f$ has the unobserved variable $X_2$ as its argument, eliminating $X_2$ requires inverting $f$. Crucially, only local invertibility is needed, so identifiability can hold for functions such as $f(X_2)=\sin(X_2)$ that are not globally invertible. A numerical example, where such a function is recovered from synthetic data, is provided in Section~\ref{si:noninvertible_recovery}.


\subsection{Two-parameter mutual activation loop model} \label{sec:diffalg_two_parameter}

We next allow the two variables to have different decay rates, giving
\begin{align}
    \dot{X}_1 &= f(X_2) - cX_1, \label{eq:diffalg_model2_dXdt} \\
    \dot{X}_2 &= X_1 - dX_2,   \label{eq:diffalg_model2_dYdt}
\end{align}
where $f:\mathbb{R}\to\mathbb{R}$ is an unknown function and $c$ and $d$ are unknown parameters.

\begin{proposition} \label{prop:diffalg_two_parameter_identifiability}
In the two-parameter mutual activation model, Equations~\eqref{eq:diffalg_model2_dXdt}--\eqref{eq:diffalg_model2_dYdt}, the identifiability properties are as follows.
\renewcommand{\labelenumi}{(\roman{enumi})}
\begin{enumerate}
    \item If only $X_1$ is observable then, provided $f$ is locally invertible and $d \neq 0$: if $f$ is non-affine, then $f$, $c$, and $d$ are structurally identifiable; and if $f$ is affine, then $c$, $d$, and $f$ are structurally non-identifiable.
    \item If only $X_2$ is observable then $f$ is structurally non-identifiable and only the quantities $c + d$ and $f(X_2) - cdX_2$ are identifiable. Specifically, given any admissible triplet $(f_1, c_1, d_1)$, any alternative choice of the form
\begin{align}
    \label{eq:diffalg_case4_nonident_expr}
    c_2 &\in\mathbb{R}, \notag  \\
    d_2    &= c_1 + d_1 - c_2, \notag  \\
    f_2(X_2) &= f_1(X_2) + (c_2 c_1 + c_2 d_1 - c_2^2 - c_1 d_1)X_2,
\end{align}
    yields dynamics identical to $(f_1, c_1, d_1)$ with respect to $X_2$.
\end{enumerate}
\end{proposition}

The proof of part~(i) is given in Section~\ref{si:proof_case3}, while part~(ii) is proved below.

\begin{proof}[Proof of part (ii)]
To eliminate $X_1$, we first isolate it in Equation~\eqref{eq:diffalg_model2_dYdt} and differentiate:
\begin{equation*}
    X_1= \dot{X}_2 + dX_2 \quad \Longrightarrow \quad \dot{X}_1 = \ddot{X}_2 + d\dot{X}_2.
\end{equation*}
Substituting into Equation~\eqref{eq:diffalg_model2_dXdt} gives the following input--output equation
\begin{equation}
    \label{eq:diffalg_model2_Yobs_Yonly}
    \ddot{X}_2 = -(c+d)\dot{X}_2 + f(X_2) - cdX_2.
\end{equation}
Now suppose two triplets $(f_1, c_1, d_1)$ and $(f_2, c_2, d_2)$ yield identical $X_2$ dynamics. Inserting both into Equation~\eqref{eq:diffalg_model2_Yobs_Yonly}, isolating $\ddot{X}_2$, and then equating gives
\begin{equation}
    \label{eq:diffalg_model2_Yobs_equiveq}
    -(c_1+d_1)\dot{X}_2 + f_1(X_2) - c_1 d_1 X_2 = \ddot{X}_2 = -(c_2+d_2)\dot{X}_2 + f_2(X_2) - c_2 d_2 X_2.
\end{equation}
Here, we can separate Equation~\eqref{eq:diffalg_model2_Yobs_equiveq} into two independent equalities:
\begin{align}
    c_1 + d_1 &= c_2 + d_2, \label{eq:diffalg_case4_sum} \\
    f_1(X_2) - c_1 d_1 X_2 &= f_2(X_2) - c_2 d_2 X_2. \label{eq:diffalg_case4_fY}
\end{align}
For any $c_2 \in \mathbb{R}$, setting $d_2 = c_1 + d_1 - c_2$ satisfies Equation~\eqref{eq:diffalg_case4_sum}. Equation~\eqref{eq:diffalg_case4_fY} then requires
\begin{equation*}
    f_2(X_2) = f_1(X_2) + (c_2 d_2 - c_1 d_1)X_2 = f_1(X_2) + \bigl(c_2(c_1 + d_1 - c_2) - c_1 d_1\bigr)X_2,
\end{equation*}
which is the expression in Equation~\eqref{eq:diffalg_case4_nonident_expr}. Since $c_2$ may be chosen freely, infinitely many valid triplets $(f_2, c_2, d_2)$ exist, hence $f$ (as well as $c$ and $d$) is structurally non-identifiable. The identifiable quantities can be determined directly from Equation~\eqref{eq:diffalg_model2_Yobs_Yonly}. Here, any triplets $(f_1,c_1,d_1)$ and $(f_2,c_2,d_2)$ that preserve the quantities $c+d$ and $f(x) - cdx$ yield identical input--output equations, preventing the identification of more detailed dynamics.
\end{proof}

We note that the non-identifiability relationship described in Equation~\eqref{eq:diffalg_case4_nonident_expr} only exhibits entanglement between the unknown function and a linear term, i.e. all of the function's higher-order dynamics can still be identified. This is the same as in Figure~\ref{fig:single_var_nonident}, where the entanglement was also limited to the linear dynamics.

Case (ii) of Proposition~\ref{prop:diffalg_two_parameter_identifiability} is the only structurally non-identifiable case among the four mutual-activation-loop examples considered in this section, which illustrates that functional identifiability depends on both the model structure and the choice of observable(s). The non-identifiability here can, nevertheless, be resolved by incorporating additional information. For example, if only $X_2$ is observed but the initial conditions of both $X_1$ and $X_2$ are known, then all unknown functions and parameters become identifiable (proof in Section~\ref{si:proof_case4_known_initial_conditions}). 

Finally, to illustrate the theoretical identifiability patterns derived for this model, we reproduce them numerically (Figure~\ref{fig:mutual_activation_numerics}). Here, we generate a single synthetic dataset (Section~\ref{si:method_numerics_data}), to which we fit the unknown components ($c$, $d$, and $f$) of the model in Equations~\eqref{eq:diffalg_model2_dXdt}--\eqref{eq:diffalg_model2_dYdt}, considering the cases where either only $X_1$ or only $X_2$ is measured (Section~\ref{si:method_numerics_fitting}). When $X_1$ is measured, $f$ is correctly recovered (Figure~\ref{fig:mutual_activation_numerics}C). When $X_2$ is measured, by contrast, a wide range of forms for $f$ reproduce the data equally well (Figure~\ref{fig:mutual_activation_numerics}F).

\begin{figure}
    \centering
    \includegraphics[width=0.98\linewidth]{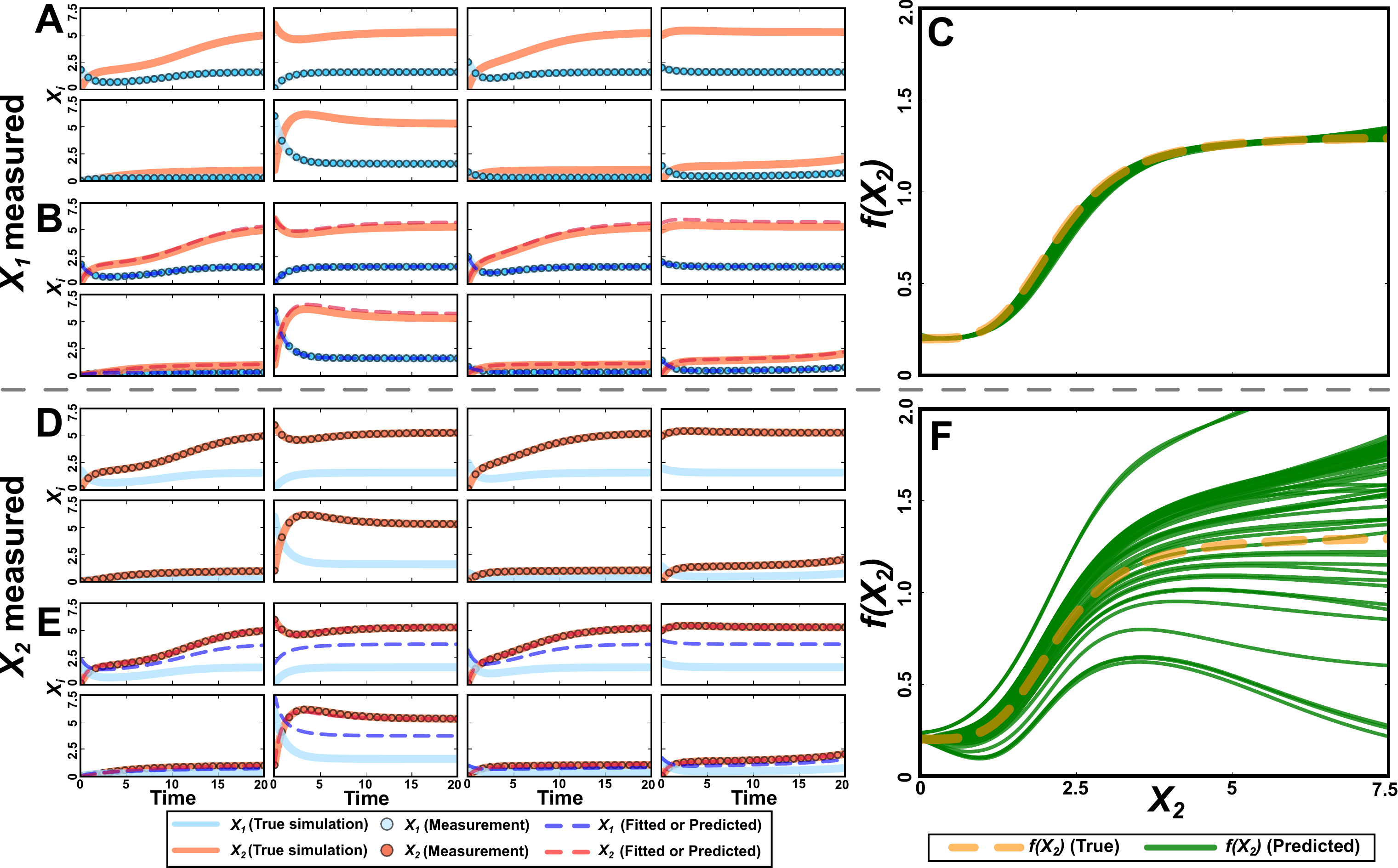}
    \caption{\textbf{The two-parameter mutual activation loop model's unknown function can only be recovered when $X_1$ is measured.} (A,D) For the two-parameter mutual activation loop model described in Equations~\eqref{eq:diffalg_model2_dXdt}--\eqref{eq:diffalg_model2_dYdt}, we designate values for the unknown components ($c =0.8$, $d = 0.3$, and $f(x) = 0.2 + 1.1x^4/(x^4 + 2.2^4)$). From this model instance, using eight different initial conditions (Table~\ref{table:mutual_activation_numerics_data}), we generate synthetic datasets according to the method detailed in Section~\ref{si:method_numerics_data}. The ground-truth simulations (faint lines) and data (dots) are shown for the cases where either only $X_1$ (subfigure A) or only $X_2$ (subfigure D) is measured. (B,E) The fitted $X_1$ and $X_2$ trajectories (dashed lines), added to the datasets in A and D. In both cases, the trajectory corresponding to the observed variable fits the data well (the overlap with the dots makes the two hard to distinguish). Only when $X_1$ is observed, however, does the predicted trajectory of the unobserved variable align with the ground-truth simulations. In E (where only $X_2$ is measured), the predictions for $X_1$ are erroneous, a consequence of the non-identifiability described in Proposition~\ref{prop:diffalg_two_parameter_identifiability}, which prevents us from correctly identifying $X_1(t)$. (C,F) We repeat the fits described in B and E. For each of the $25$ best fits, we plot the recovered $f(x)$ (green lines) against the true function (dashed yellow lines). When $X_1$ is measured (C), only functions aligning with the ground-truth function can reproduce the data, indicating identifiability. When $X_2$ is measured (F), by contrast, a wide range of functions yields good model-to-data fits, indicating non-identifiability. Indeed, the recovered forms roughly follow the ground-truth function with an added linear term, again as predicted by Proposition~\ref{prop:diffalg_two_parameter_identifiability}.}
    \label{fig:mutual_activation_numerics}
\end{figure}


\section{Different forms of structural functional non-identifiability} \label{sec:nonidentifiability_types}

In classical structural parameter identifiability, non-identifiability arises when multiple parameter combinations generate identical observable dynamics. Proposition~\ref{prop:diffalg_two_parameter_identifiability}(ii) showed that this idea extends naturally to entanglement between unknown functions and parameters, where changes in one can be compensated by changes in the other. Here, we demonstrate that functional identifiability introduces further non-identifiability types that are absent from the classical parametric setting. We also revisit the distinction between local and global identifiability for unknown functions. A proof for each proposition is given in Section~\ref{si:proofs_section4}.

\subsection{Function-to-function non-identifiability} \label{sec:nonidentifiability_types_function2function}

Consider the model
\begin{align}
    \dot{X}_1 &= f(X_2) + dX_1, \label{eq:f2f_nonident_model_dXdt} \\
    \dot{X}_2 &= X_1 - g(X_2), \label{eq:f2f_nonident_model_dYdt}
\end{align}
where $f:\mathbb{R}\to\mathbb{R}$ and $g:\mathbb{R}\to\mathbb{R}$ are unknown functions and $d\in\mathbb{R}$ is an unknown parameter.

\begin{proposition} \label{prop:f2f_nonident}
If only $X_2$ is observable, then $d$ and the unknown functions $f$ and $g$ are all structurally non-identifiable. Specifically, given any admissible triplet $(d_1,f_1, g_1)$, any alternative choice of the form
\begin{align}
    \label{eq:f2f_nonident_expr}
    d_2&\in\mathbb{R},  \notag \\
    g_2(X_2) &= g_1(X_2) + (d_2 - d_1)X_2, \notag \\
    f_2(X_2) &= f_1(X_2) + d_1 g_1(X_2) - d_2 g_1(X_2) - d_2(d_2 - d_1)X_2,
\end{align}
yields dynamics identical to $(d_1,f_1, g_1)$ with respect to $X_2$.
\end{proposition}

Here we have a non-identifiability relationship involving multiple functions: a change in $g$ propagates to a compensating change in $f$ that preserves the dynamics. While not surprising, this illustrates a general point: whereas classical non-identifiability relationships involve parameters only, the introduction of unknown functions permits relationships between any combination of parameters and functions.


\subsection{Intrinsic functional non-identifiability} \label{sec:nonidentifiability_types_pure}

The previous examples all involved an unknown function co-occurring with other unknowns. We now show that a function can be non-identifiable even in the absence of any other unknown quantities. Consider the following model with a sole unknown, $f:\mathbb{R}\to\mathbb{R}$:
\begin{align}
    \dot{X}_1 &= f(X_2) - X_1,   \label{eq:fonly_nonident_model_dXdt} \\
    \dot{X}_2 &= (X_1 - 1)X_2.   \label{eq:fonly_nonident_model_dYdt}
\end{align}

\begin{proposition} \label{prop:pure_nonident}
If only $X_1$ is observed, then the unknown function $f$ is structurally non-identifiable due to a scaling non-identifiability. Specifically, given any admissible function $f_1$, any alternative function of the form
\begin{equation} \label{eq:pure_nonident_expr}
    f_2(X_2) = f_1(X_2/k), \qquad k\in\mathbb{R}\setminus\{0\},
\end{equation}
yields identical dynamics to those of $f_1$ with respect to $X_1$.
\end{proposition}

This demonstrates another form of structural functional non-identifiability, further illustrating the additional nuance that the transition from parameters to functions introduces.


\subsection{Global and local structural identifiability} \label{sec:local_global}
In Definition~\ref{def:functional_identifiability}, we noted that identifiability can be classified as either \emph{global} or \emph{local}. This distinction is well-established for classical structural parameter identifiability, an example of which can be found in the following self-activating gene regulatory network
\begin{equation*}
    \dot{X} = \frac{X^2}{X^2 + K^2} - dX.
\end{equation*}
Here, $d$ is globally identifiable, whereas $K$ is only locally identifiable. Specifically, the quantity $K^2$ is uniquely determined by the dynamics, but one cannot distinguish between $K$ and $-K$, both of which produce identical dynamics. Additional constraints, such as restricting to $K>0$, can restore global identifiability.

The same global/local distinction applies to functions. Consider the model
\begin{equation}
    \label{eq:local_global_model}
    \dot{X} = f(X) + \frac{1}{f(X)},
\end{equation}
where $f:\mathbb{R}\to\mathbb{R}$ is an unknown function.

\begin{proposition} \label{prop:local_global}
The unknown function $f$ is locally but not globally structurally identifiable. Specifically, given any admissible function $f_1$, the alternative function
\begin{equation}
    f_2(X) = \frac{1}{f_1(X)},
\end{equation}
yields identical dynamics, so $f$ is determined, pointwise, only up to the two-element equivalence class $\{f,\; 1/f\}$. If $f$ is known to be continuous or smooth (reasonable assumptions), the pointwise choices are further constrained according to the switching mechanism described in Figure~\ref{fig:local_ident_switching}. 
\end{proposition}

Further local functional ambiguities can arise from sign symmetries, argument symmetries, and label-swapping between multiple functions; examples of each are given in Section~\ref{si:local_identifiability_examples}.


\section{Applications} \label{sec:applications}

We now apply the framework to two model classes with direct real-world relevance, a chemical reaction network model and the Lotka--Volterra model. In each case, we describe when the model exhibits structural identifiability and non-identifiability, and the form of any non-identifiabilities. Proofs are given in Section~\ref{si:proofs_section5}.


\subsection{Chemical reaction networks} \label{sec:applications_crn}

Chemical reaction networks are a widely used modelling framework in biology and chemistry~\cite{hahl_comparison_2016,feinberg_foundations_2019}, with applications ranging from systems biology~\cite{goodwin_oscillatory_1965} and epidemiology~\cite{avram_advancing_2024} to pharmacology~\cite{zou_application_2020} and chemical kinetics~\cite{field_oscillations_1974,warnatz_combustion_1999}. Ordinary differential equation models are typically derived from reaction networks via the law of mass action. Importantly, the resulting systems often contain parameters or functional relationships that appear in multiple equations. Such repeated occurrences introduce additional constraints that can substantially affect identifiability. In this section, we use a chemical reaction network model both to illustrate structural functional identifiability analysis for this important class of systems and to demonstrate how identifiability is influenced when an unknown function appears across multiple equations.

We consider a simple two-species network with ODE model given by 
\begin{align}
    \dot{X}_1 &= f(X_2) - dX_1, \label{eq:crn_dXdt} \\
    \dot{X}_2 &= X_1 - f(X_2), \label{eq:crn_dYdt}
\end{align}
where $f:\mathbb{R}\to\mathbb{R}$ is an unknown function and $d>0$ is an unknown parameter. The derivation of the ODE from a chemical reaction model is described in Section~\ref{si:crn_derivation}.

\begin{proposition} \label{prop:crn}
When only $X_2$ is observable, $f$ and $d$ are structurally identifiable, provided $d \neq 1$ and $f$ is non-affine. If $d = 1$, then $d$ and $f'$, but not $f$, are identifiable. When $f$ is affine, $f$ and $d$ are locally identifiable.
\end{proposition}

When $f$ is affine, identifiability is only local: writing $f(x) = ax + b$, the combinations $a + d$ and $(1-d)a$ are globally identifiable, but $a$ and $d$ individually admit two solutions (the unordered pair $\{a, d-1\}$ is recoverable, but their permutation is not). Full identifiability, therefore, requires a non-affine $f$.

\begin{proposition} \label{prop:crn_Xobs}
When only $X_1$ is observable, $d$ is structurally identifiable (assuming $f$ is locally invertible and non-affine) while $f$ is structurally non-identifiable. Specifically, given any admissible triplet $(f_1, d, X_{2,1}(0))$, any alternative choice of the form
\begin{align}
    \label{eq:crn_Xobs_nonident_expr}
    k&\in\mathbb{R}, \notag \\
    f_2(X_2) &= f_1(X_2 + k), \notag \\
    X_{2,2}(0) &= X_{2,1}(0) - k,
\end{align}
yields dynamics identical to $(f_1, d, X_{2,1}(0))$ with respect to $X_1$. Here, $X_2(0)$ is the initial value of the unobserved variable $X_2$.
\end{proposition}

Again, this example demonstrates how identifiability depends on which variables are observable. It also demonstrates a case of intrinsic functional non-identifiability, i.e.~one that does not involve entanglement with non-identifiable parameters. Finally, it illustrates the relevance of information about initial conditions: if, in addition to $X_1$, the initial condition of $X_2$ were known, full identifiability would be achieved. 


\subsection{The Lotka--Volterra model} \label{sec:applications_LV}

The Lotka--Volterra model is a classic example of predator--prey population dynamics~\cite{Murray_MathematicalBiology_2002}. The model has also seen extensive use to demonstrate hybrid modelling workflows, primarily for universal differential equations~\cite{dandekar_bayesian_2022,fronk_bayesian_2024,giampiccolo_robust_2024,grigorian_learning_2024,rackauckas_universal_2021}. Furthermore, it is an example where the unknown function depends on multiple variables, introducing additional complexity. We consider two variants, which differ in how the interaction term between the two species is represented. For each, we determine identifiability when either one or both of the species are observable.

\paragraph{Model 1} The two interaction terms are represented by distinct functions:
\begin{align}
    \dot{X}_1 &= \alpha X_1 + f(X_1,X_2), \label{eq:LV1_dXdt} \\
    \dot{X}_2 &= g(X_1,X_2) + \delta X_2, \label{eq:LV1_dYdt}
\end{align}
where $X_1,X_2\in[0,\infty)$, $f:[0,\infty)^2\to(-\infty,0]$ and $g:[0,\infty)^2\to[0,\infty)$ are unknown functions that are zero whenever $X_1=0$ or $X_2=0$ and non-zero otherwise, and $\alpha\in(0,\infty)$ and $\delta\in(-\infty,0)$ are unknown parameters.

\begin{proposition} \label{prop:LV1_identifiability}
For Model~1, the following hold.
\renewcommand{\labelenumi}{(\roman{enumi})}
\begin{enumerate}
    \item When only $X_1$ is observable, $\alpha$ is structurally identifiable, while $\delta$, $f$ and $g$ are structurally non-identifiable. Specifically, given any admissible triplet $(f_1,g_1,\delta_1)$, any alternative choice of the form
    \begin{align*}
        \lambda &\in (0,\infty), \\
        \delta_2 &\in (-\infty,0), \\
        f_2(X_1,X_2) &= f_1\Bigl(X_1,\, \bigl(X_2/\lambda\bigr)^{\delta_1/\delta_2}\Bigr), \\
        g_2(X_1,X_2) &= \frac{\delta_2}{\delta_1}\,\lambda^{\delta_1/\delta_2}X_2^{1-\delta_1/\delta_2}\, g_1\Bigl(X_1,\, \bigl(X_2/\lambda\bigr)^{\delta_1/\delta_2}\Bigr),
    \end{align*}
    with $g_2(X_1,0)=0$, yields dynamics identical to $(f_1,g_1,\delta_1)$ with respect to $X_1$. Additionally, the dynamics of the unobserved variable $X_2$ can be determined up to the alternative form
    \begin{equation*}
        X_{2,2}(t) = \lambda X_{2,1}(t)^{\delta_2/\delta_1}.
    \end{equation*}
    \item When only $X_2$ is observable, $\delta$ is structurally identifiable, while $\alpha$, $f$ and $g$ are structurally non-identifiable. Specifically, given any admissible triplet $(f_1,g_1,\alpha_1)$, any alternative choice of the form
    \begin{align*}
        \lambda &\in (0,\infty), \\
        \alpha_2 &\in (0,\infty), \\
        f_2(X_1,X_2) &= \frac{\alpha_2}{\alpha_1}\,\lambda^{\alpha_1/\alpha_2}X_1^{1-\alpha_1/\alpha_2}\, f_1\Bigl(\bigl(X_1/\lambda\bigr)^{\alpha_1/\alpha_2},\, X_2\Bigr), \\
        g_2(X_1,X_2) &= g_1\Bigl(\bigl(X_1/\lambda\bigr)^{\alpha_1/\alpha_2},\, X_2\Bigr),
    \end{align*}
    with $f_2(0,X_2)=0$, yields dynamics identical to $(f_1,g_1,\alpha_1)$ with respect to $X_2$. Additionally, the dynamics of the unobserved variable $X_1$ can be determined up to the alternative form
    \begin{equation*}
        X_{1,2}(t) = \lambda X_{1,1}(t)^{\alpha_2/\alpha_1}.
    \end{equation*}
    \item When both $X_1$ and $X_2$ are observable, $\alpha$, $\delta$, $f$ and $g$ are all structurally identifiable.
\end{enumerate}
\end{proposition}

While Model~1 has the form of the fully augmented systems in Section~\ref{sec:nonident_complete_UDEs}, Proposition~\ref{prop:complete_UDE_nonident} does not apply here, since its alternatives, $f_2 = f_1 + (\alpha_1-\alpha_2)X_1$ and $g_2 = g_1 + (\delta_1-\delta_2)X_2$, do not vanish when a species is absent. This shows how additional knowledge of the modelled system can constrain it enough to restore identifiability.

\paragraph{Model 2} In contrast to Model~1, the two interaction terms are assumed to be identical up to a linear rescaling:
\begin{align}
    \dot{X}_1 &= \alpha X_1 + f(X_1,X_2),        \label{eq:LV2_dXdt} \\
    \dot{X}_2 &= \gamma f(X_1,X_2) + \delta X_2, \label{eq:LV2_dYdt}
\end{align}
where $X_1,X_2\in[0,\infty)$, $f:[0,\infty)^2\to(-\infty,0]$ is an unknown function that is zero whenever $X_1=0$ or $X_2=0$ and non-zero otherwise, and $\alpha\in(0,\infty)$, $\gamma\in(-\infty,0)$ and $\delta\in(-\infty,0)$ are unknown parameters.

\begin{proposition} \label{prop:LV2_observability}
For Model~2, the following hold.
\renewcommand{\labelenumi}{(\roman{enumi})}
\begin{enumerate}
    \item When only $X_1$ is observable, $\alpha$ is structurally identifiable, while $f$ and $\gamma$ are structurally non-identifiable. Specifically, given any admissible pair $(f_1,\gamma_1)$, any alternative choice of the form
    \begin{align*}
        \gamma_2 &\in (-\infty,0), \\
        f_2(X_1,X_2) &= f_1\Bigl(X_1,\, \frac{\gamma_1}{\gamma_2}X_2\Bigr),
    \end{align*}
    yields dynamics identical to $(f_1,\gamma_1)$ with respect to $X_1$. Additionally, the dynamics of the unobserved variable $X_2$ can be determined up to the alternative form
    \begin{equation*}
        X_{2,2}(t) = \frac{\gamma_2}{\gamma_1}X_{2,1}(t).
    \end{equation*}
    \item When only $X_2$ is observable, $\delta$ is structurally identifiable, while $f$ and $\gamma$ are structurally non-identifiable. Specifically, given any admissible pair $(f_1,\gamma_1)$, any alternative choice of the form
    \begin{align*}
        \gamma_2 &\in (-\infty,0), \\
        f_2(X_1,X_2) &= \frac{\gamma_1}{\gamma_2}\, f_1\Bigl(\frac{\gamma_2}{\gamma_1}X_1,\, X_2\Bigr),
    \end{align*}
    yields dynamics identical to $(f_1,\gamma_1)$ with respect to $X_2$. Additionally, the dynamics of the unobserved variable $X_1$ can be determined up to the alternative form
    \begin{equation*}
        X_{1,2}(t) = \frac{\gamma_1}{\gamma_2}X_{1,1}(t).
    \end{equation*}
    \item When both $X_1$ and $X_2$ are observable, $f$, $\alpha$, $\gamma$ and $\delta$ are all structurally identifiable.
\end{enumerate}
\end{proposition}

For the Lotka--Volterra model, we complement the analysis by considering whether the dynamics of the unobserved variables can be identified (termed \emph{observability}, the analogue of identifiability for state variables). For brevity, we have not given observability much attention throughout this work; however, the Lotka--Volterra example shows that it can be treated similarly, e.g. for the Lotka--Volterra model, we obtain entanglements for non-observable variables, as we typically do for functions. Here, full identifiability requires both species to be observed. However, we note that alternative formulations of the Lotka--Volterra model can achieve identifiability, for example by assuming certain parameters are known.


\section{Discussion} \label{sec:discussion}

In this work, we have extended the concept of structural identifiability from scalar parameters to unknown functions. Although functional and parametric identifiability are conceptually related, and can both be analysed using differential algebra-based approaches, our results show that the functional setting introduces phenomena that have no direct analogue in the classical parameter identifiability framework. In addition to entanglement between unknown functions and parameters, we identified non-identifiability relationships involving function scaling, interactions between multiple unknown functions, and dependencies on unobserved system dynamics. These examples illustrate that moving from parameter inference to function inference is not simply a change in dimensionality, but introduces qualitatively new forms of ambiguity.

The transition from parameters to functions also affects the identifiability analysis itself. While the differential algebra approach remains applicable, the resulting derivations often require additional mathematical considerations that do not arise in the parametric setting. For example, the proof of Proposition~\ref{prop:diffalg_one_parameter_identifiability}(ii) relies on properties of the inverse of the unknown function, and several other examples require reasoning about functional transformations rather than parameter substitutions. Functional identifiability therefore extends not only the scope of structural identifiability analysis, but also the range of mathematical properties that such analyses must accommodate, such as whether an unknown function is locally invertible or affine.

The framework demonstrated here is largely agnostic to the particular hybrid modelling methodology employed. It applies directly to approaches in which unknown functions are embedded within differential equations, including universal differential equations, physics-informed neural networks, and biologically-informed neural networks~\cite{rackauckas_universal_2021,raissi_physics-informed_2019,lagergren_biologically-informed_2020,karniadakis_physics-informed_2021}. The framework can, in theory, also extend to equation discovery methods, such as SINDy or symbolic regression-based approaches~\cite{schmidt_distilling_2009,brunton_discovering_2016}. More broadly, any modelling approach that seeks to infer unknown functional components within a fixed dynamical system falls within the scope of the theory developed here. By contrast, extending structural functional identifiability to settings in which the model structure itself changes, for example through the introduction of additional state variables and governing equations, remains an open challenge.

Although structural parameter identifiability is now a well-established field, its scope continues to expand. Recent work has extended identifiability analysis beyond ordinary differential equations to settings including partial differential equations and stochastic differential equations~\cite{renardy_structural_2022,browning_structural_2024,byrne_algebraic_2025,browning_exact_2025}. An important direction for future research is, therefore, to determine how functional identifiability can be formulated and analysed in these broader modelling frameworks~\cite{loman_learning_2026}. Further theoretical development is also needed to better accommodate more general classes of unknown functions, including functions of multiple variables, independent variables such as time, and unknown functions that depend on unknown parameters.

From a methodological perspective, applying a differential algebra-based approach to unknown functions is often more involved than in the classical parametric setting. While the examples considered here remain tractable, it is not yet clear how well the approach scales to larger or more complex models. Furthermore, the widespread adoption of structural parameter identifiability has been facilitated by the development of automated software tools~\cite{bellu_daisy_2007,structidjl}. Determining whether comparable automation is possible for structural functional identifiability, and developing the corresponding computational infrastructure, will be essential if functional identifiability analysis is to be applied routinely across scientific disciplines.

As scientific modelling becomes increasingly data-driven, a central question is what role mechanistic knowledge should play. Hybrid models offer a principled way to combine domain expertise with flexible function inference, but their flexibility also introduces the risk that multiple explanations may fit the same observations equally well. Structural functional identifiability provides a means of determining when mechanistic assumptions genuinely constrain the recovery of unknown dynamics and when they do not. In this sense, it offers a theoretical foundation for assessing whether hybrid models deliver new scientific understanding or simply alternative representations of the same observed behaviour.


\section*{Code availability}

All scripts for regenerating all figures (as well as all other numerical analyses) can be found here: \url{https://github.com/TorkelE/StructFuncIdent-Loman-Browning-Baker-2026}.


\appendix

\section{Proofs} \label{appendix:proofs_section2}
Here we provide the proofs for the propositions in Section~\ref{sec:general_nonident}. The proofs for the remaining propositions, which almost exclusively consist of applications of the differential algebra approach, are provided in the supplementary material.

\subsection{Proof of Proposition~\ref{prop:1var_nonident}} \label{appendix:proof_1var}

Given an admissible choice $(f_1,\vec{p}_1)$, we wish to show that any proposed choice of the form
\begin{align*}    
    \vec{p}_2 &\in \mathbb{R}^{\operatorname{dim}(\vec{p})}, \\
    f_2(X) &= f_1(X) + g(X;\vec{p}_1) - g(X;\vec{p}_2),
\end{align*}
yields identical dynamics. Inserting these expressions into Equation~\eqref{eq:nonident_1var_UDEs} yields
\begin{align*}
    \dot{X} &= f_2(X) + g(X;\vec{p}_2) \\
            &= f_1(X) + g(X;\vec{p}_1) - g(X;\vec{p}_2) + g(X;\vec{p}_2) \\
            &= f_1(X) + g(X;\vec{p}_1).
\end{align*}
This is identical to the dynamics generated by $(f_1,\vec{p}_1)$. Since $\vec{p}_2$ may be chosen freely, infinitely many valid $f_2$ exist, hence $f$ (as well as $\vec{p}$) is structurally non-identifiable.

\subsection{Proof of Proposition~\ref{prop:multivar_nonident}} \label{appendix:proof_multivar}

Given an admissible choice with parameter split $\vec{p}_{1,1}=\{\vec{a}_1,\vec{b}\}$, we wish to show that any proposed choice $\vec{p}_{1,2}=\{\vec{a}_2,\vec{b}\}$ of the form
\begin{align*}
    \vec{a}_2    &\in \mathbb{R}^{\operatorname{dim}(\vec{a})}, \\
    f_2(\vec{u}) &= f_1(\vec{u}) + r(\vec{v};\vec{a}_1) - r(\vec{v};\vec{a}_2),
\end{align*}
yields identical dynamics. Inserting these expressions into Equation~\eqref{eq:nonident_partial_system_UDEs} yields
\begin{align*}
    \dot{X}_1 &= f_2(\vec{u}) + g_1(\vec{X};\vec{p}_{1,2}) \\
              &= f_2(\vec{u}) + r(\vec{v};\vec{a}_2) + s(\vec{X};\vec{b}) \\
              &= f_1(\vec{u}) + r(\vec{v};\vec{a}_1) - r(\vec{v};\vec{a}_2) + r(\vec{v};\vec{a}_2) + s(\vec{X};\vec{b}) \\
              &= f_1(\vec{u}) + r(\vec{v};\vec{a}_1) + s(\vec{X};\vec{b}) \\
              &= f_1(\vec{u}) + g_1(\vec{X};\vec{p}_{1,1}).
\end{align*}
This is identical to the dynamics generated by $(f_1, \vec{p}_{1,1})$. The remaining equations are unchanged, as neither $f$ nor $\vec{a}$ appears in these. Since $\vec{a}_2$ may be chosen freely, infinitely many valid choices of $f_2$ exist, hence $f$ is structurally non-identifiable.

\subsection{Repeated-function counterexample} \label{appendix:repeated_function_counterexample}

The proof of Proposition~\ref{prop:multivar_nonident} requires that $f$ appears only in the first equation. To see why this matters, consider a repeated-function extension
\begin{align}
    \dot{X}_1 &= f(\vec{u}) + g_1(\vec{X};\vec{p}_1), \label{eq:repeated_f_counter_X1} \\
    \dot{X}_2 &= f(\vec{u}) + g_2(\vec{X};\vec{p}_2), \label{eq:repeated_f_counter_X2} \\
    \dot{X}_i &= g_i(\vec{X};\vec{p}_i), \qquad i=3,\ldots,n, \notag
\end{align}
where $f$ appears in both the first and second equation. Here, identifiability may still hold for $f$.

\begin{proof}
The construction in Proposition~\ref{prop:multivar_nonident} changes $\vec{a}_1$ to an arbitrary $\vec{a}_2$ and compensates in the first equation by setting
\begin{equation*}
    f_2(\vec{u}) = f_1(\vec{u}) + r(\vec{v};\vec{a}_1) - r(\vec{v};\vec{a}_2).
\end{equation*}
This still preserves Equation~\eqref{eq:repeated_f_counter_X1} because
\begin{align*}
    f_2(\vec{u}) + r(\vec{v};\vec{a}_2) + s(\vec{X};\vec{b})
    = f_1(\vec{u}) + r(\vec{v};\vec{a}_1) + s(\vec{X};\vec{b}) = f_1(\vec{u}) + g_1(\vec{X};\vec{p}_1).
\end{align*}
However, the same substitution changes Equation~\eqref{eq:repeated_f_counter_X2}:
\begin{align*}
    f_2(\vec{u}) + g_2(\vec{X};\vec{p}_2)
    = f_1(\vec{u}) + r(\vec{v};\vec{a}_1) - r(\vec{v};\vec{a}_2) + g_2(\vec{X};\vec{p}_2).
\end{align*}
This equation is identical to the original Equation~\eqref{eq:repeated_f_counter_X2} only if $r(\vec{v};\vec{a}_1)=r(\vec{v};\vec{a}_2)$ on the admissible domain. Thus the arbitrary choice of $\vec{a}_2$ that generates the non-identifiable family in Proposition~\ref{prop:multivar_nonident} is no longer available. A trivial case where $f$ remains identifiable is where $g_2$ depends on no unknown parameters, i.e.~$\operatorname{dim}(\vec{p}_2) = 0$. In this case, Equation~\eqref{eq:repeated_f_counter_X2} directly determines $f(\vec{u})$ from the observed dynamics (without even needing to consider Equation~\eqref{eq:repeated_f_counter_X1}). This does not prove that every repeated-function model is identifiable, but it shows that Proposition~\ref{prop:multivar_nonident} cannot be extended unchanged once $f$ appears in another equation.
\end{proof}

\subsection{Proof of Proposition~\ref{prop:complete_UDE_nonident}} \label{appendix:proof_complete_UDE}

Given an admissible choice $(f_{i,1}, \vec{p}_{i,1})$, choose any consistent alternative parameter sets $\vec{p}_{i,2}$, respecting any overlaps between the $\vec{p}_i$, and define
\begin{align*}
    f_{i,2}(\vec{X}) &= f_{i,1}(\vec{X}) + g_i(\vec{X};\vec{p}_{i,1}) - g_i(\vec{X};\vec{p}_{i,2}).
\end{align*}
Inserting these into Equation~\eqref{eq:complete_UDE} gives
\begin{align*}
    \dot{X}_i &= f_{i,2}(\vec{X}) + g_i(\vec{X};\vec{p}_{i,2}) \\
              &= f_{i,1}(\vec{X}) + g_i(\vec{X};\vec{p}_{i,1}) - g_i(\vec{X};\vec{p}_{i,2}) + g_i(\vec{X};\vec{p}_{i,2}) \\
              &= f_{i,1}(\vec{X}) + g_i(\vec{X};\vec{p}_{i,1}).
\end{align*}
This is identical to the dynamics generated by $(f_{i,1}, \vec{p}_{i,1})$. Since the alternative parameter sets $\vec{p}_{i,2}$ may be varied consistently, infinitely many valid choices of $f_{i,2}$ exist, hence all $f_i$ (as well as all $\vec{p}_i$) are structurally non-identifiable.

\subsection{Proof of Proposition~\ref{prop:nonident_neural_odes}} \label{appendix:proof_neural_odes}

\textit{If all states are observed, the system is identifiable.} Since each $X_i$ is fully observed, both it and its derivatives are fully known. From this, and the relation $\dot{X}_i = f_i(\vec{X})$, we have that each $f_i$ is fully identifiable.

\medskip\noindent
\textit{If at least one state is unobserved, the system is non-identifiable.} We will show that non-identifiability holds for the case where only a single variable, $X_1$, is unobservable. From this, it directly follows that non-identifiability still holds when observability for additional states is lost.

Take an admissible set of functions $f_{j,1}$ and an associated trajectory of the unobserved state $X_{1,1}(t)$. We wish to show that any proposed choice of the form
\begin{align*}
    k    &\in \mathbb{R}, \\
    f_{i,2}(X_1, X_2, \ldots) &= f_{i,1}(X_1 - k, X_2, \ldots),\\
    X_{1,2}(t) &= X_{1,1}(t) + k,
\end{align*}
yields identical dynamics. Inserting these into Equation~\eqref{eq:nonident_neural_odes} gives, for each $i$,
\begin{align*}
    \dot{X}_i &= f_{i,2}(X_{1,2}, X_2, \ldots) \\
              &= f_{i,1}(X_{1,2} - k, X_2, \ldots) \\
              &= f_{i,1}(X_{1,1} + k - k, X_2, \ldots) \\
              &= f_{i,1}(X_{1,1}, X_2, \ldots).
\end{align*}
This is identical to the dynamics generated by $(f_{i,1}, X_{1,1}(t))$. Since $k$ may be chosen freely, infinitely many valid choices of $f_{i,2}$ exist, hence all $f_i$ (as well as the dynamics of the unknown variable $X_1$) are structurally non-identifiable.


\section*{Acknowledgments}

TEL and REB are supported by a grant from the Simons Foundation (MP-SIP-00001828). For the purpose of open access, the author has applied a CC BY public copyright licence to any author-accepted manuscript arising from this submission. AI tools were used in the writing of this article and the development of the corresponding code. Specifically, all text and code were first written by the authors, after which AI tools were used for editing and finding errors. All final text and code were then verified, understood, and edited by the authors before publication.

\bibliographystyle{plainnat}
\bibliography{references}


\clearpage
\setcounter{section}{0}
\renewcommand{\thesection}{S\arabic{section}}
\renewcommand{\theHsection}{S\arabic{section}}

\begin{center}
    {\Large\bfseries Supplementary material}
\end{center}

\section{Proofs} \label{si:proofs}

Proofs for the results of Section~\ref{sec:general_nonident} are given in Appendix~\ref{appendix:proofs_section2}. Here, we give the proofs for the results of Sections~\ref{sec:diffalg}--\ref{sec:applications} (which almost exclusively consist of applications of the differential algebra approach demonstrated in the main article).


\subsection{Proofs for Section~\ref{sec:diffalg}} \label{si:proofs_section3}

For Proposition~\ref{prop:diffalg_one_parameter_identifiability}, part~(i) is proved in the main text and part~(ii) below. For Proposition~\ref{prop:diffalg_two_parameter_identifiability}, part~(i) is proved below and part~(ii) in the main text.

\subsubsection{Proof of Proposition~\ref{prop:diffalg_one_parameter_identifiability}(ii)} \label{si:proof_case2}

Isolate $X_2$ from Equation~\eqref{eq:diffalg_model1_dXdt}, writing
\begin{equation*}
    X_2 = f^{-1}(\dot{X}_1 + dX_1).
\end{equation*}
Let $u = \dot{X}_1 + dX_1$ and $g = f^{-1}$, so that $X_2 = g(u)$ and $\dot{X}_2 = g'(u)\dot{u}$. Substituting into Equation~\eqref{eq:diffalg_model1_dYdt} and rearranging gives
\begin{equation}
    \label{eq:diffalg_model1_Xobs_Xeq}
    \ddot{X}_1 = \frac{X_1 - dg(u)}{g'(u)} - d\dot{X}_1.
\end{equation}
Suppose $(f_1,d_1)$ and $(f_2,d_2)$ yield identical $X_1$ dynamics. Writing $g_1 = f_1^{-1}$, $g_2 = f_2^{-1}$, $u_1 = \dot{X}_1 + d_1 X_1$, and $u_2 = \dot{X}_1 + d_2 X_1$ and inserting both into Equation~\eqref{eq:diffalg_model1_Xobs_Xeq} gives
\begin{equation}
    \label{eq:diffalg_model1_Xobs_equiveq}
    \frac{X_1 - d_1g_1(u_1)}{g_1'(u_1)} - d_1\dot{X}_1 = \ddot{X}_1 = \frac{X_1 - d_2g_2(u_2)}{g_2'(u_2)} - d_2\dot{X}_1.
\end{equation}
Here, we treat $X_1$ and $\dot{X}_1$ as independent variables. From this, and Lemma~\ref{lemma:g_prime_equality} (shown below), we have that either $d_1 = d_2$ or both $g_1'$ and $g_2'$ are constant.

\medskip\noindent
\textit{Case 1: $d_1 = d_2 \equiv d$.} Then $u_1 = u_2 \equiv u$. Treating $X_1$ and $u$ as independent variables, comparing the coefficients of $X_1$ in Equation~\eqref{eq:diffalg_model1_Xobs_equiveq} gives $g_1'(u) = g_2'(u)$, and the remaining terms then give $g_1(u) = g_2(u)$ (as $d \neq 0$), and hence $f_1 = f_2$. The issue of invertibility of $f$ is discussed in Remark~\ref{remark:proof_invertibility}.

\medskip\noindent
\textit{Case 2: $g_1'$ and $g_2'$ constant.} By Lemma~\ref{lemma:constant_g_prime} below, $d_1 = d_2$ and $g_1 = g_2$, hence $f_1 = f_2$.

\medskip\noindent
In both cases, $d_1 = d_2$ and $f_1 = f_2$, so both $d$ and $f$ are structurally identifiable (subject to the local invertibility and $d \neq 0$ conditions considered below).

\begin{remark} \label{remark:proof_invertibility}

Multiple proofs of identifiability throughout this work (specifically those where the unknown function depends on an unobserved variable) assume that $f$ can be inverted. Here, the identifiability arguments hold as long as $f$ is \emph{locally invertible} (i.e.~does not contain any flat regions).
\begin{proof}
Local invertibility implies that $f' = 0$ only at isolated points. Hence, the domain of $f$ can be written as a countable union of maximal open intervals $\bigcup_k I_k$ on each of which $f$ is strictly monotone. On each $I_k$, a smooth local inverse $g_k\colon f(I_k) \to I_k$ exists. The main proof can be applied to each $I_k$ independently because the input--output equation is a pointwise relation: it constrains the values of $f$ and $f'$ at a point $y$, without coupling to the values of $f$ at any other point. Since $f|_{I_k}$ is invertible, the main proof establishes identifiability of $f$ on each $I_k$ and of all parameters. The parameter conclusions are constants, so they agree across intervals. The function identity $f_1 = f_2$ therefore holds on $\bigcup_k I_k$, which is dense in $\mathbb{R}$. Since $f_1 - f_2$ is continuous and vanishes on a dense set, $f_1 = f_2$ everywhere.
\end{proof}
\end{remark}

\begin{lemma} \label{lemma:g_prime_equality}
Let $u_1 = \dot{X}_1 + d_1 X_1$, $u_2 = \dot{X}_1 + d_2 X_1$ with $d_1, d_2 \neq 0$. If Equation~\eqref{eq:diffalg_model1_Xobs_equiveq} holds for all admissible values of $X_1$ and $\dot{X}_1$, then either $d_1 = d_2$ (and hence $u_1 = u_2$), or both $g_1'$ and $g_2'$ are constant functions.
\end{lemma}

\begin{proof}
If $d_1 = d_2$, then $u_1 = u_2$ and the conclusion holds trivially. Suppose therefore that $d_1 \neq d_2$; we show both $g_1'$ and $g_2'$ must be constant. Then $u_1$ and $u_2$ can be treated as independent variables, with $X_1 = (u_1 - u_2)/(d_1 - d_2)$ and $\dot{X}_1 = u_1 - d_1X_1$. Writing $h_i = 1/g_i'$, Equation~\eqref{eq:diffalg_model1_Xobs_equiveq} becomes
\begin{equation*}
    h_1(u_1)\bigl(X_1 - d_1g_1(u_1)\bigr) - h_2(u_2)\bigl(X_1 - d_2g_2(u_2)\bigr) = (d_1 - d_2)\dot{X}_1 = d_1u_2 - d_2u_1.
\end{equation*}
Differentiating once with respect to $u_1$ and once with respect to $u_2$ gives $-\bigl(h_1'(u_1) + h_2'(u_2)\bigr)/(d_1 - d_2) = 0$, so $h_1'$ is a constant, $s$ say, and $h_1(u) = su + r$. Differentiating twice with respect to $u_1$ then gives $2s/(d_1 - d_2) - d_1(h_1g_1)''(u_1) = 0$. Since $h_1g_1' = 1$, $(h_1g_1)'' = sg_1' = s/(su_1 + r)$, so
\begin{equation*}
    s\left(\frac{2}{d_1 - d_2} - \frac{d_1}{su_1 + r}\right) = 0
\end{equation*}
for all values of $u_1$. As $d_1 \neq 0$, this requires $s = 0$, so $h_1$, and hence $g_1'$, is constant. Exchanging the roles of the two choices shows that $g_2'$ is constant as well, proving the lemma.
\end{proof}

\begin{lemma} \label{lemma:constant_g_prime}
Under the hypotheses of Lemma~\ref{lemma:g_prime_equality}, and assuming both pairs satisfy Equation~\eqref{eq:diffalg_model1_Xobs_equiveq}, if both $g_1'$ and $g_2'$ are constant, then $d_1 = d_2$ and $g_1 = g_2$.
\end{lemma}

\begin{proof}
Write $g_1(x) = a_1x + b_1$ and $g_2(x) = a_2x + b_2$. Substituting into Equation~\eqref{eq:diffalg_model1_Xobs_Xeq} gives, for each choice,
\begin{equation*}
    \ddot{X}_1 = -2d_i\dot{X}_1 + \left(\frac{1}{a_i} - d_i^2\right)X_1 - \frac{d_ib_i}{a_i}.
\end{equation*}
Equating coefficients gives $d_1 = d_2$ (from the $\dot{X}_1$ term), then $a_1 = a_2$, and hence $b_1 = b_2$ (since $d \neq 0$), so $g_1 = g_2$.
\end{proof}

\subsubsection{Proof of Proposition~\ref{prop:diffalg_two_parameter_identifiability}(i)} \label{si:proof_case3}

Isolate $X_2$ from Equation~\eqref{eq:diffalg_model2_dXdt} by writing
\begin{equation*}
    X_2 = f^{-1}(\dot{X}_1 + cX_1).
\end{equation*}
Set $u = \dot{X}_1 + cX_1$ and $g = f^{-1}$, so that $X_2 = g(u)$ and $\dot{X}_2 = g'(u)\dot{u}$ with $\dot{u} = \ddot{X}_1 + c\dot{X}_1$. Substituting into Equation~\eqref{eq:diffalg_model2_dYdt} and rearranging gives
\begin{equation}
    \label{eq:proofs_Xobs_f_ident_main}
    \ddot{X}_1 = \frac{X_1 - dg(u)}{g'(u)} - c\dot{X}_1.
\end{equation}
Suppose $(f_1,c_1,d_1)$ and $(f_2,c_2,d_2)$ yield identical $X_1$ dynamics. Write $g_1 = f_1^{-1}$, $g_2 = f_2^{-1}$, $u_1 = \dot{X}_1 + c_1 X_1$, and $u_2 = \dot{X}_1 + c_2 X_1$. Inserting both into Equation~\eqref{eq:proofs_Xobs_f_ident_main} gives
\begin{equation}
    \label{eq:proofs_Xobs_f_ident_equiv}
    \frac{X_1 - d_1g_1(u_1)}{g_1'(u_1)} - c_1\dot{X}_1 = \ddot{X}_1 = \frac{X_1 - d_2g_2(u_2)}{g_2'(u_2)} - c_2\dot{X}_1.
\end{equation}
Here, we treat $X_1$ and $\dot{X}_1$ as independent variables. The argument of Lemma~\ref{lemma:g_prime_equality}, with $u_i = \dot{X}_1 + c_iX_1$, gives that either $c_1 = c_2$ or both $g_1'$ and $g_2'$ are constant (which implies that $g = f^{-1}$ is affine).

\medskip\noindent
\textit{Case 1: $c_1 = c_2 \equiv c$.} Then $u_1 = u_2 \equiv u$. Treating $X_1$ and $u$ as independent variables, comparing the coefficients of $X_1$ in Equation~\eqref{eq:proofs_Xobs_f_ident_equiv} gives $g_1'(u) = g_2'(u)$, and the remaining terms then give $d_1 g_1(u) = d_2 g_2(u)$. Differentiating gives $d_1 g_1'(u) = d_2 g_2'(u)$. Since $g_1' = g_2'$, either $d_1 = d_2$ (giving $g_1 = g_2$ and full identifiability) or $g_1' = g_2' = 0$ (so $g$ is constant, which contradicts local invertibility of $f$).

\medskip\noindent
\textit{Case 2: $g_1'$ and $g_2'$ constant.} Write $g_1(x) = a_1x + b_1$ and $g_2(x) = a_2x + b_2$. Substituting into Equation~\eqref{eq:proofs_Xobs_f_ident_main} gives, for each choice, $\ddot{X}_1 = -(c_i + d_i)\dot{X}_1 + (1/a_i - c_id_i)X_1 - d_ib_i/a_i$. Comparing coefficients gives
\begin{align*}
    c_1 + d_1 &= c_2 + d_2, \\
    1/a_1 - c_1 d_1 &= 1/a_2 - c_2 d_2, \\
    b_1 d_1/a_1 &= b_2 d_2/a_2.
\end{align*}
These are three relations for the four unknowns. For any $c_2$, setting $d_2 = c_1 + d_1 - c_2$ and choosing $a_2$ and $b_2$ to satisfy the second and third relations yields an identical input--output equation, so $c$, $d$, and $f$ are structurally non-identifiable.

\medskip\noindent
In Case~1, global identifiability holds. In Case~2, $f^{-1}$ must be affine, and $c$, $d$, and $f$ are structurally non-identifiable. Together, these cases establish part~(i) of Proposition~\ref{prop:diffalg_two_parameter_identifiability}.

\subsubsection{Identifiability with known initial conditions for the two-parameter mutual activation model} \label{si:proof_case4_known_initial_conditions}

\begin{proof}
Consider the model in Equations~\eqref{eq:diffalg_model2_dXdt}--\eqref{eq:diffalg_model2_dYdt} with only $X_2$ observed, but suppose that $X_1(0)$ and $X_2(0)$ are known for each observed trajectory. Assume generic initial conditions, so that at least one trajectory has $X_2(0)\neq 0$. Since $X_2$ is observed, $\dot{X}_2(0)$ is also known. Evaluating Equation~\eqref{eq:diffalg_model2_dYdt} at $t=0$ gives
\begin{equation*}
    \dot{X}_2(0)=X_1(0)-dX_2(0),
\end{equation*}
and hence
\begin{equation*}
    d=\frac{X_1(0)-\dot{X}_2(0)}{X_2(0)}.
\end{equation*}
Thus $d$ is structurally identifiable.

Now suppose two choices $(f_1,c_1,d)$ and $(f_2,c_2,d)$ yield identical $X_2$ dynamics. Substituting both into the input--output equation, Equation~\eqref{eq:diffalg_model2_Yobs_Yonly}, gives
\begin{equation*}
    -(c_1+d)\dot{X}_2 + f_1(X_2)-c_1dX_2
    =
    -(c_2+d)\dot{X}_2 + f_2(X_2)-c_2dX_2.
\end{equation*}
Treating $X_2$ and $\dot{X}_2$ as independent variables, the coefficient of $\dot{X}_2$ gives $c_1+d=c_2+d$, and therefore $c_1=c_2$. Substituting this back into the same equation gives $f_1(X_2)=f_2(X_2)$. Hence $d$, $c$, and $f$ are all structurally identifiable.
\end{proof}


\subsection{Proofs for Section~\ref{sec:nonidentifiability_types}} \label{si:proofs_section4}

\subsubsection{Proof of Proposition~\ref{prop:f2f_nonident}} \label{si:proof_f2f}

Isolate $X_1$ from Equation~\eqref{eq:f2f_nonident_model_dYdt} and differentiate to give
\begin{align*}
    X_1       &= \dot{X}_2 + g(X_2), \\
    \dot{X}_1 &= \ddot{X}_2 + g'(X_2)\dot{X}_2.
\end{align*}
Substituting into Equation~\eqref{eq:f2f_nonident_model_dXdt} and rearranging we have
\begin{align}
    \label{eq:f2f_nonident_dY2dt2}
    \ddot{X}_2 + g'(X_2)\dot{X}_2 &= f(X_2) + d\bigl(\dot{X}_2 + g(X_2)\bigr) \notag \Longleftrightarrow \\
    \ddot{X}_2 &= \dot{X}_2(d - g'(X_2)) + f(X_2) + dg(X_2).
\end{align}
Suppose $(f_1, g_1, d_1)$ and $(f_2, g_2, d_2)$ yield identical $X_2$ dynamics. Selecting $f_2$, $g_2$, $d_2$ according to the forms in Equation~\eqref{eq:f2f_nonident_expr} and noting that this yields $g_2'(X_2) = g_1'(X_2) + (d_2 - d_1)$, we can insert into the right-hand side of Equation~\eqref{eq:f2f_nonident_dY2dt2} to give
\begin{align*}
    &\dot{X}_2(d_2 - g_2'(X_2)) + f_2(X_2) + d_2 g_2(X_2) \\
    & \qquad = \dot{X}_2(d_2 - g_1'(X_2) - (d_2 - d_1)) \\
    &\qquad \qquad  + f_1(X_2) + d_1 g_1(X_2) - d_2 g_1(X_2) - d_2(d_2-d_1)X_2 + d_2 g_1(X_2) + d_2(d_2-d_1)X_2 \\
    &\qquad = \dot{X}_2(d_1 - g_1'(X_2)) + f_1(X_2) + d_1 g_1(X_2).
\end{align*}
This expression equals the right-hand side of Equation~\eqref{eq:f2f_nonident_dY2dt2} for the first choice. Hence, the second-choice values yield identical $X_2$ dynamics. Since $d_2$ may be chosen freely and each choice produces distinct $f_2 \neq f_1$ and $g_2 \neq g_1$, $d$, $f$, and $g$ are all structurally non-identifiable.

\subsubsection{Proof of Proposition~\ref{prop:pure_nonident}} \label{si:proof_pure}

Given an admissible function $f_1$ with associated trajectories $X_1(t)$ and $X_{2,1}(t)$, take $f_2$ as in Equation~\eqref{eq:pure_nonident_expr} for any $k \neq 0$, and set $X_{2,2}(t) = kX_{2,1}(t)$. Inserting these into Equations~\eqref{eq:fonly_nonident_model_dXdt}--\eqref{eq:fonly_nonident_model_dYdt} gives
\begin{align*}
    \dot{X}_1 &= f_2(X_{2,2}) - X_1 = f_1(X_{2,1}) - X_1, \\
    \dot{X}_{2,2} &= k\dot{X}_{2,1} = k(X_1 - 1)X_{2,1} = (X_1 - 1)X_{2,2}.
\end{align*}
Thus the second choice reproduces the $X_1$ dynamics of the first. Since $k$ may be chosen freely, $f$ is structurally non-identifiable.

\subsubsection{Proof of Proposition~\ref{prop:local_global}} \label{si:proof_local_global}

Substituting $f_2(X) = 1/f_1(X)$ into the right-hand side of Equation~\eqref{eq:local_global_model} gives
\begin{equation*}
    f_2(X) + \frac{1}{f_2(X)} = \frac{1}{f_1(X)} + \frac{1}{1/f_1(X)} = \frac{1}{f_1(X)} + f_1(X) = f_1(X) + \frac{1}{f_1(X)}.
\end{equation*}
The right-hand side is identical for $f_1$ and $f_2$, so both pairs yield identical $X$ dynamics. Since $f_2 = 1/f_1 \neq f_1$ whenever $f_1(X)^2 \neq 1$ for some $X$, the function $f$ is not globally identifiable. Finally, $(1/f_2)(X) = f_1(X)$, so the equivalence class contains exactly two elements, confirming local identifiability. Furthermore, even if $f$ is assumed to be continuous, whichever of the relations $f_1=f_2$ and $f_1=1/f_2$ holds may switch at points where $f_1(x^*)=f_2(x^*)$, permitting the generation of a potentially large equivalence class of functions. This phenomenon is described in Figure~\ref{fig:local_ident_switching}.


\subsection{Proofs for Section~\ref{sec:applications}} \label{si:proofs_section5}

\subsubsection{Proof of Proposition~\ref{prop:crn}} \label{si:proof_crn}

Isolate $X_1$ from Equation~\eqref{eq:crn_dYdt} and differentiate to give
\begin{align*}
    X_1       &= \dot{X}_2 + f(X_2), \\
    \dot{X}_1 &= \ddot{X}_2 + f'(X_2)\dot{X}_2.
\end{align*}
Substituting into Equation~\eqref{eq:crn_dXdt} we have
\begin{equation}
    \label{eq:crn_Yobs_dY2dt2}
    \ddot{X}_2 = (1-d)f(X_2) - (d + f'(X_2))\dot{X}_2.
\end{equation}
Suppose $(f_1,d_1)$ and $(f_2,d_2)$ yield identical $X_2$ dynamics. Inserting both into Equation~\eqref{eq:crn_Yobs_dY2dt2} and equating gives
\begin{equation}
    \label{eq:crn_Yobs_equiveq}
    (1-d_1)f_1(X_2) - (d_1 + f_1'(X_2))\dot{X}_2 = \ddot{X}_2 = (1-d_2)f_2(X_2) - (d_2 + f_2'(X_2))\dot{X}_2.
\end{equation}
Here, treating $X_2$ and $\dot{X}_2$ as independent variables allows us to separate Equation~\eqref{eq:crn_Yobs_equiveq} into two independent equalities:
\begin{align}
    (1-d_1)f_1(X_2) &= (1-d_2)f_2(X_2), \label{eq:crn_Yobs_dalg_1} \\
    d_1 + f_1'(X_2) &= d_2 + f_2'(X_2). \label{eq:crn_Yobs_dalg_2}
\end{align}

\medskip\noindent
\textit{Case 1: $f$ is non-affine.} Suppose for contradiction that $d_1 \neq d_2$. From Equation~\eqref{eq:crn_Yobs_dalg_1} (if $d_1 = 1$, this gives $f_2 \equiv 0$, contradicting that $f$ is non-affine, so we may assume $d_1 \neq 1$):
\begin{equation*}
    f_1(X_2) = \frac{1-d_2}{1-d_1} f_2(X_2).
\end{equation*}
Differentiating and substituting into Equation~\eqref{eq:crn_Yobs_dalg_2} gives $f_1'(X_2) = d_2 - 1$, a constant, so $f_1$ is affine. By the same argument applied to $f_2$, $f_2$ is also affine. This contradicts our assumption that $f$ is non-affine. Therefore $d_1 = d_2 \equiv d$. We then have:
\begin{itemize}
    \item If $d \neq 1$, Equation~\eqref{eq:crn_Yobs_dalg_1} gives $f_1 = f_2$, and so both $d$ and $f$ are globally structurally identifiable.
    \item If $d = 1$, Equation~\eqref{eq:crn_Yobs_dalg_1} is trivially satisfied for any $f_1$, $f_2$ and Equation~\eqref{eq:crn_Yobs_dalg_2} gives $f_1' = f_2'$. Thus $d$ is identifiable, but only $f'$ (not $f$ itself) is identifiable.
\end{itemize}

\medskip\noindent
\textit{Case 2: $f$ is affine.} Write $f_i(X_2) = a_i X_2 + b_i$. Substituting into Equations~\eqref{eq:crn_Yobs_dalg_1}--\eqref{eq:crn_Yobs_dalg_2} and equating coefficients gives
\begin{align*}
    (1-d_1)a_1 &= (1-d_2)a_2, \\
    (1-d_1)b_1 &= (1-d_2)b_2, \\
    d_1 + a_1  &= d_2 + a_2.
\end{align*}
The first and third relations show that $a + d$ and $(1-d)a$ are globally identifiable or, equivalently, that $a + (d-1)$ and $a(d-1)$ are. They are therefore symmetric with respect to $a$ and $d-1$. Hence we can only identify the two values of the set $\{a, d-1\}$, but not assign values to the individual parameters. The second relation then determines $b$ only up to the same two-fold ambiguity, giving the alternative solution
\begin{equation*}
    a_2 = d_1 - 1, \qquad d_2 = a_1 + 1, \qquad b_2 = \frac{(d_1-1)b_1}{a_1}.
\end{equation*}
Thus $a$, $b$, and $d$ are each locally identifiable.

In summary, full structural identifiability holds only in Case~1 with $d \neq 1$.

\subsubsection{Proof of Proposition~\ref{prop:crn_Xobs}} \label{si:proof_crn_Xobs}

Given an admissible choice $(f_1, d)$ with associated unobserved trajectory $X_{2,1}(t)$, we wish to show that the family in Equation~\eqref{eq:crn_Xobs_nonident_expr} yields identical $X_1$ dynamics. Setting $X_{2,2}(t) = X_{2,1}(t) - k$, so that $X_{2,2}(0) = X_{2,1}(0) - k$ and $\dot{X}_{2,2} = \dot{X}_{2,1}$, and substituting $f_2(X_2) = f_1(X_2+k)$ into Equations~\eqref{eq:crn_dXdt}--\eqref{eq:crn_dYdt} gives
\begin{align*}
    \dot{X}_1 &= f_2(X_{2,2}) - dX_1 = f_1(X_{2,2} + k) - dX_1 = f_1(X_{2,1}) - dX_1, \\
    \dot{X}_{2,2} &= X_1 - f_2(X_{2,2}) = X_1 - f_1(X_{2,1}) = \dot{X}_{2,1}.
\end{align*}
Thus the second choice reproduces the first-choice $X_1$ dynamics exactly. Since $k$ may be chosen freely, $f$ is structurally non-identifiable.

We now show that $d$ is identifiable (provided $f$ is locally invertible and non-affine). Set $g = f^{-1}$ and $u = \dot{X}_1 + dX_1$. From $\dot{X}_1 = f(X_2) - dX_1$ we have $X_2 = g(u)$, and differentiating gives $\dot{X}_2 = g'(u)\dot{u}$ with $\dot{u} = \ddot{X}_1 + d\dot{X}_1$. Substituting into Equation~\eqref{eq:crn_dYdt} and rearranging yields
\begin{align} 
    g'(u)\dot{u} &= X_1 - f(g(u)) \Longleftrightarrow \notag \\
    g'(u)\dot{u} &= X_1 - u \Longleftrightarrow \notag \\
    g'(u)(\ddot{X}_1 + d\dot{X}_1) &= X_1 - (\dot{X}_1 + dX_1) \Longleftrightarrow \notag \\
    \ddot{X}_1 &= \frac{(1-d)X_1 - \dot{X}_1}{g'(u)} - d\dot{X}_1. \label{eq:crn_Xobs_Xdot}
\end{align}
Suppose $(g_1, d_1)$ and $(g_2, d_2)$ both yield identical $X_1$ dynamics, with $u_i = \dot{X}_1 + d_iX_1$. Inserting each into Equation~\eqref{eq:crn_Xobs_Xdot} and equating (since both expressions equal the same $\ddot{X}_1$) gives
\begin{equation*}
    \frac{(1-d_1)X_1 - \dot{X}_1}{g_1'(u_1)} - d_1\dot{X}_1 = \ddot{X}_1 = \frac{(1-d_2)X_1 - \dot{X}_1}{g_2'(u_2)} - d_2\dot{X}_1.
\end{equation*}
Here, we treat $X_1$ and $\dot{X}_1$ as independent variables. Suppose $d_1 \neq d_2$. Then $u_1$ and $u_2$ can also be treated as independent variables, with $X_1 = (u_1 - u_2)/(d_1 - d_2)$ and $\dot{X}_1 = u_1 - d_1X_1$, so that $(1-d_i)X_1 - \dot{X}_1 = X_1 - u_i$. Writing $h_i = 1/g_i'$, substituting, and multiplying by $\delta = d_1 - d_2$ gives
\begin{equation*}
    h_1(u_1)\bigl((1-\delta)u_1 - u_2\bigr) - h_2(u_2)\bigl(u_1 - (1+\delta)u_2\bigr) = \delta\,(d_1u_2 - d_2u_1).
\end{equation*}
Differentiating once with respect to $u_1$ and once with respect to $u_2$ gives $-h_1'(u_1) - h_2'(u_2) = 0$, so $h_1(u) = su + r_1$ and $h_2(u) = -su + r_2$. The left-hand side then contains the terms $s(1-\delta)u_1^2 - s(1+\delta)u_2^2$, which must vanish; since $\delta$ cannot be both $1$ and $-1$, $s = 0$. Then $g_1'$ is constant, so $f$ is affine, which is excluded. Hence $d_1 = d_2$, i.e. $d$ is structurally identifiable.

\subsubsection{Proofs for the Lotka--Volterra model} \label{si:proof_LV}

\begin{proof}[Proof of Proposition~\ref{prop:LV1_identifiability}(i)]
We first prove that $\alpha$ is structurally identifiable. Suppose $(f_1,g_1,\alpha_1,\delta_1)$ and $(f_2,g_2,\alpha_2,\delta_2)$ yield identical dynamics with respect to $X_1$, with unobserved trajectories $X_{2,1}(t)$ and $X_{2,2}(t)$. Inserting both into Equation~\eqref{eq:LV1_dXdt} and equating gives
\begin{equation*}
    \alpha_1X_1 + f_1(X_1,X_{2,1}) = \dot{X}_1 = \alpha_2X_1 + f_2(X_1,X_{2,2}).
\end{equation*}
Here, both $f_1$ and $f_2$ depend on the unobserved variable, so the two sides cannot be compared term by term. Instead, we compare the values of $(X_1,\dot{X}_1)$ that each side can produce. Since $f_1\le0$, the left-hand side gives $\dot{X}_1\le\alpha_1X_1$ for all observed values. For the second choice, Equation~\eqref{eq:LV1_dYdt} gives $\dot{X}_{2,2}=g_2(X_1,0)=0$ whenever $X_{2,2}=0$, so trajectories with $X_{2,2}(t)=0$ exist. Along these the right-hand side gives $\dot{X}_1=\alpha_2X_1$ with $X_1>0$. Since both choices produce the same observed values, $\alpha_2X_1\le\alpha_1X_1$, i.e.\ $\alpha_2\le\alpha_1$. Exchanging the roles of the two choices gives $\alpha_1\le\alpha_2$. Hence $\alpha_1=\alpha_2\equiv\alpha$, and $\alpha$ is structurally identifiable.

Next, given an admissible triplet $(f_1,g_1,\delta_1)$ with associated unobserved trajectory $X_{2,1}(t)$, we wish to show that any choice of the following form
\begin{align*}
    \lambda &\in (0,\infty), \\
    \delta_2 &\in (-\infty,0), \\
    f_2(X_1,X_{2,2}) &= f_1(X_1,X_{2,1}), \\
    g_2(X_1,X_{2,2}) &= \frac{\delta_2}{\delta_1}\,\frac{X_{2,2}}{X_{2,1}}\,g_1(X_1,X_{2,1}), \\
    X_{2,2}(t) &= \lambda X_{2,1}(t)^{\delta_2/\delta_1},
\end{align*}
which is the family of the proposition written in terms of $X_{2,1}$, yields dynamics identical to $(f_1,g_1,\delta_1,X_{2,1}(t))$. Inserting $(f_2,X_{2,2}(t))$ into Equation~\eqref{eq:LV1_dXdt} gives
\begin{align*}
    \dot{X}_1 &= \alpha X_1 + f_2(X_1,X_{2,2}) \\
              &= \alpha X_1 + f_1(X_1,X_{2,1}).
\end{align*}
The $X_1$ dynamics of $(f_1,g_1,\delta_1,X_{2,1}(t))$ are therefore reproduced. It remains to verify that $X_{2,2}$ satisfies Equation~\eqref{eq:LV1_dYdt} for the second model. Since $X_1$ and $X_{2,1}$ satisfy the first system, we have
\begin{align*}
    \dot{X}_{2,2} &= \frac{\delta_2}{\delta_1}\,\lambda X_{2,1}^{\delta_2/\delta_1-1}\,\dot{X}_{2,1} \\
                  &= \frac{\delta_2}{\delta_1}\,\frac{X_{2,2}}{X_{2,1}}\bigl(g_1(X_1,X_{2,1}) + \delta_1X_{2,1}\bigr) \\
                  &= \frac{\delta_2}{\delta_1}\,\frac{X_{2,2}}{X_{2,1}}\,g_1(X_1,X_{2,1}) + \delta_2X_{2,2} \\
                  &= g_2(X_1,X_{2,2}) + \delta_2X_{2,2}.
\end{align*}
Thus $(f_2,g_2,\delta_2,X_{2,2})$ is a valid solution to the system that reproduces the same $X_1$ dynamics. It is admissible, since $f_2$ and $g_2$ have the signs and zeros of $f_1$ and $g_1$. Since $X_2\mapsto\lambda X_2^{\delta_2/\delta_1}$ maps $[0,\infty)$ onto itself, every trajectory of the second choice arises in this way. Since $\lambda$ and $\delta_2$ may be chosen freely, infinitely many valid choices exist, so $\delta$, $f$ and $g$ are structurally non-identifiable.
\end{proof}

\begin{proof}[Proof of Proposition~\ref{prop:LV1_identifiability}(ii)]
We first prove that $\delta$ is structurally identifiable. Suppose $(f_1,g_1,\alpha_1,\delta_1)$ and $(f_2,g_2,\alpha_2,\delta_2)$ yield identical dynamics with respect to $X_2$, with unobserved trajectories $X_{1,1}(t)$ and $X_{1,2}(t)$. Inserting both into Equation~\eqref{eq:LV1_dYdt} and equating gives
\begin{equation*}
    g_1(X_{1,1},X_2) + \delta_1X_2 = \dot{X}_2 = g_2(X_{1,2},X_2) + \delta_2X_2.
\end{equation*}
Since $g_1\ge0$, the left-hand side gives $\dot{X}_2\ge\delta_1X_2$ for all observed values. For the second choice, Equation~\eqref{eq:LV1_dXdt} gives $\dot{X}_{1,2}=f_2(0,X_2)=0$ whenever $X_{1,2}=0$, so trajectories with $X_{1,2}(t)=0$ exist, and along these the right-hand side gives $\dot{X}_2=\delta_2X_2$ with $X_2>0$. Since both choices produce the same observed values, $\delta_2\ge\delta_1$. Exchanging the roles of the two choices gives $\delta_1\ge\delta_2$. Hence $\delta_1=\delta_2\equiv\delta$, and $\delta$ is structurally identifiable.

Next, given an admissible triplet $(f_1,g_1,\alpha_1)$ with associated unobserved trajectory $X_{1,1}(t)$, we wish to show that any choice of the following form
\begin{align*}
    \lambda &\in (0,\infty), \\
    \alpha_2 &\in (0,\infty), \\
    f_2(X_{1,2},X_2) &= \frac{\alpha_2}{\alpha_1}\,\frac{X_{1,2}}{X_{1,1}}\,f_1(X_{1,1},X_2), \\
    g_2(X_{1,2},X_2) &= g_1(X_{1,1},X_2), \\
    X_{1,2}(t) &= \lambda X_{1,1}(t)^{\alpha_2/\alpha_1},
\end{align*}
which is the family of the proposition written in terms of $X_{1,1}$, yields dynamics identical to $(f_1,g_1,\alpha_1,X_{1,1}(t))$. Inserting $(g_2,X_{1,2}(t))$ into Equation~\eqref{eq:LV1_dYdt} gives
\begin{align*}
    \dot{X}_2 &= g_2(X_{1,2},X_2) + \delta X_2 \\
              &= g_1(X_{1,1},X_2) + \delta X_2.
\end{align*}
The $X_2$ dynamics of $(f_1,g_1,\alpha_1,X_{1,1}(t))$ are therefore reproduced. It remains to verify that $X_{1,2}$ satisfies Equation~\eqref{eq:LV1_dXdt} for the second model. Since $X_{1,1}$ and $X_2$ satisfy the first system, we have
\begin{align*}
    \dot{X}_{1,2} &= \frac{\alpha_2}{\alpha_1}\,\lambda X_{1,1}^{\alpha_2/\alpha_1-1}\,\dot{X}_{1,1} \\
                  &= \frac{\alpha_2}{\alpha_1}\,\frac{X_{1,2}}{X_{1,1}}\bigl(\alpha_1X_{1,1} + f_1(X_{1,1},X_2)\bigr) \\
                  &= \alpha_2X_{1,2} + \frac{\alpha_2}{\alpha_1}\,\frac{X_{1,2}}{X_{1,1}}\,f_1(X_{1,1},X_2) \\
                  &= \alpha_2X_{1,2} + f_2(X_{1,2},X_2).
\end{align*}
Thus $(f_2,g_2,\alpha_2,X_{1,2})$ is a valid solution to the system that reproduces the same $X_2$ dynamics. It is admissible, since $f_2$ and $g_2$ have the signs and zeros of $f_1$ and $g_1$. Since $X_1\mapsto\lambda X_1^{\alpha_2/\alpha_1}$ maps $[0,\infty)$ onto itself, every trajectory of the second choice arises in this way. Since $\lambda$ and $\alpha_2$ may be chosen freely, infinitely many valid choices exist, so $\alpha$, $f$ and $g$ are structurally non-identifiable.
\end{proof}

\begin{proof}[Proof of Proposition~\ref{prop:LV1_identifiability}(iii)]
Suppose $(f_1,g_1,\alpha_1,\delta_1)$ and $(f_2,g_2,\alpha_2,\delta_2)$ yield identical $X_1$ and $X_2$ dynamics. Inserting both into Equations~\eqref{eq:LV1_dXdt}--\eqref{eq:LV1_dYdt} and equating gives
\begin{align*}
    \alpha_1 X_1 + f_1(X_1,X_2) &= \dot{X}_1 = \alpha_2 X_1 + f_2(X_1,X_2), \\
    g_1(X_1,X_2) + \delta_1 X_2 &= \dot{X}_2 = g_2(X_1,X_2) + \delta_2 X_2,
\end{align*}
where $X_1$ and $X_2$ can be varied independently. Setting $X_2=0$ in the first equation, where $f_1$ and $f_2$ are zero, gives $\alpha_1X_1=\alpha_2X_1$, and hence $\alpha_1=\alpha_2\equiv\alpha$. With $\alpha$ known, $f(X_1,X_2)=\dot{X}_1-\alpha X_1$ is directly computable from observations, hence $f_1=f_2$. Similarly, setting $X_1=0$ in the second equation, where $g_1$ and $g_2$ are zero, gives $\delta_1=\delta_2\equiv\delta$, and $g(X_1,X_2)=\dot{X}_2-\delta X_2$ then gives $g_1=g_2$. Since there is only a single choice $(f,g,\alpha,\delta)$ yielding any given dynamics, all functions and parameters are structurally identifiable.
\end{proof}

\begin{proof}[Proof of Proposition~\ref{prop:LV2_observability}(i)]
We first prove that $\alpha$ is structurally identifiable. Suppose $(f_1,\alpha_1,\gamma_1,\delta_1)$ and $(f_2,\alpha_2,\gamma_2,\delta_2)$ yield identical dynamics with respect to $X_1$, with unobserved trajectories $X_{2,1}(t)$ and $X_{2,2}(t)$. Inserting both into Equation~\eqref{eq:LV2_dXdt} and equating gives
\begin{equation*}
    \alpha_1X_1 + f_1(X_1,X_{2,1}) = \dot{X}_1 = \alpha_2X_1 + f_2(X_1,X_{2,2}).
\end{equation*}
As in the proof of Proposition~\ref{prop:LV1_identifiability}(i), we compare the values of $(X_1,\dot{X}_1)$ that each side can produce. Since $f_1\le0$, the left-hand side gives $\dot{X}_1\le\alpha_1X_1$. For the second choice, Equation~\eqref{eq:LV2_dYdt} gives $\dot{X}_{2,2}=\gamma_2f_2(X_1,0)=0$ whenever $X_{2,2}=0$, so trajectories with $X_{2,2}(t)=0$ exist, and along these the right-hand side gives $\dot{X}_1=\alpha_2X_1$ with $X_1>0$. Hence $\alpha_2\le\alpha_1$, and exchanging the roles of the two choices gives $\alpha_1=\alpha_2\equiv\alpha$.

Next, given an admissible pair $(f_1,\gamma_1)$ with associated unobserved trajectory $X_{2,1}(t)$, we wish to show that any choice of the following form
\begin{align*}
    \gamma_2 &\in (-\infty,0), \\
    f_2(X_1,X_{2,2}) &= f_1(X_1,X_{2,1}), \\
    X_{2,2}(t) &= \frac{\gamma_2}{\gamma_1}X_{2,1}(t),
\end{align*}
which is the family of the proposition written in terms of $X_{2,1}$, yields dynamics identical to $(f_1,\gamma_1,X_{2,1}(t))$. Inserting $(f_2,X_{2,2}(t))$ into Equation~\eqref{eq:LV2_dXdt} gives
\begin{align*}
    \dot{X}_1 &= \alpha X_1 + f_2(X_1,X_{2,2}) \\
              &= \alpha X_1 + f_1(X_1,X_{2,1}).
\end{align*}
The $X_1$ dynamics of $(f_1,\gamma_1,X_{2,1}(t))$ are therefore reproduced. It remains to verify that $X_{2,2}$ satisfies Equation~\eqref{eq:LV2_dYdt} for the second model. Since $X_1$ and $X_{2,1}$ satisfy the first system, we have
\begin{align*}
    \dot{X}_{2,2} &= \frac{\gamma_2}{\gamma_1}\dot{X}_{2,1} \\
                  &= \frac{\gamma_2}{\gamma_1}\bigl(\gamma_1 f_1(X_1,X_{2,1}) + \delta X_{2,1}\bigr) \\
                  &= \gamma_2 f_1(X_1,X_{2,1}) + \delta X_{2,2} \\
                  &= \gamma_2 f_2(X_1,X_{2,2}) + \delta X_{2,2}.
\end{align*}
Thus $(f_2,\gamma_2,X_{2,2})$ is a valid solution to the full system that reproduces the same $X_1$ dynamics. It is admissible, since $f_2$ has the sign and zeros of $f_1$. Since $\gamma_2/\gamma_1>0$, the rescaling maps $[0,\infty)$ onto itself, so every trajectory of the second choice arises in this way. Since $\gamma_2$ may be chosen freely, infinitely many valid choices exist, so $f$ and $\gamma$ are structurally non-identifiable.
\end{proof}

\begin{proof}[Proof of Proposition~\ref{prop:LV2_observability}(ii)]
We first prove that $\delta$ is structurally identifiable. Suppose $(f_1,\alpha_1,\gamma_1,\delta_1)$ and $(f_2,\alpha_2,\gamma_2,\delta_2)$ yield identical dynamics with respect to $X_2$, with unobserved trajectories $X_{1,1}(t)$ and $X_{1,2}(t)$. Inserting both into Equation~\eqref{eq:LV2_dYdt} and equating gives
\begin{equation*}
    \gamma_1 f_1(X_{1,1},X_2) + \delta_1X_2 = \dot{X}_2 = \gamma_2 f_2(X_{1,2},X_2) + \delta_2X_2.
\end{equation*}
Since $\gamma_1<0$ and $f_1\le0$, the left-hand side gives $\dot{X}_2\ge\delta_1X_2$. For the second choice, Equation~\eqref{eq:LV2_dXdt} gives $\dot{X}_{1,2}=f_2(0,X_2)=0$ whenever $X_{1,2}=0$, so trajectories with $X_{1,2}(t)=0$ exist, and along these the right-hand side gives $\dot{X}_2=\delta_2X_2$ with $X_2>0$. Hence $\delta_2\ge\delta_1$, and exchanging the roles of the two choices gives $\delta_1=\delta_2\equiv\delta$.

Next, given an admissible pair $(f_1,\gamma_1)$ with associated unobserved trajectory $X_{1,1}(t)$, we wish to show that any choice of the following form
\begin{align*}
    \gamma_2 &\in (-\infty,0), \\
    f_2(X_{1,2},X_2) &= \frac{\gamma_1}{\gamma_2}f_1(X_{1,1},X_2), \\
    X_{1,2}(t) &= \frac{\gamma_1}{\gamma_2}X_{1,1}(t),
\end{align*}
which is the family of the proposition written in terms of $X_{1,1}$, yields dynamics identical to $(f_1,\gamma_1,X_{1,1}(t))$. Inserting $(f_2,\gamma_2,X_{1,2}(t))$ into Equation~\eqref{eq:LV2_dYdt} gives
\begin{align*}
    \dot{X}_2 &= \gamma_2 f_2(X_{1,2},X_2) + \delta X_2 \\
              &= \gamma_2\,\frac{\gamma_1}{\gamma_2}f_1(X_{1,1},X_2) + \delta X_2 \\
              &= \gamma_1 f_1(X_{1,1},X_2) + \delta X_2.
\end{align*}
The $X_2$ dynamics of $(f_1,\gamma_1,X_{1,1}(t))$ are therefore reproduced. It remains to verify that $X_{1,2}$ satisfies Equation~\eqref{eq:LV2_dXdt} for the second model. Since $X_{1,1}$ and $X_2$ satisfy the first system, we have
\begin{align*}
    \dot{X}_{1,2} &= \frac{\gamma_1}{\gamma_2}\dot{X}_{1,1} \\
                  &= \frac{\gamma_1}{\gamma_2}\bigl(\alpha X_{1,1} + f_1(X_{1,1},X_2)\bigr) \\
                  &= \alpha X_{1,2} + f_2(X_{1,2},X_2).
\end{align*}
Thus $(f_2,\gamma_2,X_{1,2})$ is a valid solution to the system that reproduces the same $X_2$ dynamics. Admissibility, and that every trajectory of the second choice arises in this way, follow as in part~(i). Since $\gamma_2$ may be chosen freely, infinitely many valid choices exist, so $f$ and $\gamma$ are structurally non-identifiable.
\end{proof}

\begin{proof}[Proof of Proposition~\ref{prop:LV2_observability}(iii)]
Suppose $(f_1, \alpha_1, \gamma_1, \delta_1)$ and $(f_2, \alpha_2, \gamma_2, \delta_2)$ yield identical $X_1$ and $X_2$ dynamics. Inserting both into Equations~\eqref{eq:LV2_dXdt}--\eqref{eq:LV2_dYdt} and equating gives
\begin{align*}
    \alpha_1 X_1 + f_1(X_1,X_2) &= \dot{X}_1 = \alpha_2 X_1 + f_2(X_1,X_2), \\
    \gamma_1 f_1(X_1,X_2) + \delta_1 X_2 &= \dot{X}_2 = \gamma_2 f_2(X_1,X_2) + \delta_2 X_2,
\end{align*}
where $X_1$ and $X_2$ can be varied independently. The first equation gives $f_2 = f_1 + (\alpha_1 - \alpha_2)X_1$, and substituting this into the second gives
\begin{equation*}
    (\gamma_1 - \gamma_2)f_1(X_1,X_2) = \gamma_2(\alpha_1 - \alpha_2)X_1 + (\delta_2 - \delta_1)X_2.
\end{equation*}
If $\gamma_1 \neq \gamma_2$, then $f_1$ is of the form $pX_1 + qX_2$. Since $f_1$ is zero whenever $X_1=0$ or $X_2=0$, this gives $p=q=0$, which contradicts $f_1$ being non-zero for $X_1,X_2>0$. Hence $\gamma_1 = \gamma_2 \equiv \gamma$, and since $\gamma \neq 0$, comparing the coefficients of $X_1$ and $X_2$ gives $\alpha_1 = \alpha_2$ and $\delta_1 = \delta_2$. With $\alpha$ now known, $f(X_1,X_2) = \dot{X}_1 - \alpha X_1$ is directly computable from observations, hence $f_1 = f_2$. Since there is only a single choice $(f, \alpha, \gamma, \delta)$ yielding any given dynamics, all functions and parameters are structurally identifiable.
\end{proof}


\section{Example models for Section~\ref{sec:general_nonident}} \label{si:example_models}

Below are example models satisfying the conditions of each proposition in Section~\ref{sec:general_nonident}.


\subsection{Proposition~\ref{prop:1var_nonident}: scalar ODEs} \label{si:examples_1var}

Such models are of the form $\dot{X} = f(X) + g(X;\vec{p})$ with $f$ unknown and at least one unknown parameter in $g$.

\begin{enumerate}
    
    \item \textbf{Falling object with unknown drag.} Velocity $v$ of an object subject to gravity and an unknown drag force.
    \begin{equation*}
        \dot{v} = g_0 - f(v).
    \end{equation*}
    Here $-f(v)$ is the unknown drag law, describing how air resistance depends on speed and $g(v;g_0) = g_0$ is the known gravitational term, with unknown acceleration $g_0$.

    \item \textbf{Population growth with unknown density-dependent mortality.} A population of size $N$ growing linearly, subject to an unknown density-dependent mortality capturing effects such as crowding or resource competition.
    \begin{equation*}
        \dot{N} = rN - f(N).
    \end{equation*}
    Here $-f(N)$ is the unknown density-dependent mortality law and $g(N;r) = rN$ is the known linear growth term, with unknown per-capita rate $r$.

    \item \textbf{Tumour growth under treatment.} Tumour volume $V$ with an unknown growth law and linear treatment-induced clearance.
    \begin{equation*}
        \dot{V} = f(V) - kV.
    \end{equation*}
    Here $f(V)$ is the unknown growth or proliferation law and $g(V;k) = -kV$ is the known linear treatment-clearance term, with unknown rate $k$.

    \item \textbf{Newton's law of cooling with unknown radiation.} Temperature $T$ of an object cooling via convection to an ambient environment and via an unknown radiative process.
    \begin{equation*}
        \dot{T} = f(T) - h(T - T_{\mathrm{env}}).
    \end{equation*}
    Here $f(T)$ is the unknown radiative heat-loss law, $g(T;h,T_{\mathrm{env}}) = -h(T - T_{\mathrm{env}})$ is the known convective term, with unknown heat-transfer coefficient $h$ and unknown ambient temperature $T_{\mathrm{env}}$.
    
\end{enumerate}


\subsection{Proposition~\ref{prop:multivar_nonident}: partially augmented ODEs} \label{si:examples_multivar}

 We consider the class of models where the equation containing $f$ decomposes as $\dot{X}_1 = f(\vec{u}) + r(\vec{v};\vec{a}) + s(\vec{X};\vec{b})$ with $\vec{v} \subseteq \vec{u}$ and $\vec{a}$ absent from all other equations. The key structural requirement is that the known part of the equation containing $f$ includes a parameter-dependent term $r(\vec{v};\vec{a})$ acting on variables $\vec{v}$ that are also arguments of $f$. Changes in $\vec{a}$ can always be absorbed into $f$ while leaving the dynamics unchanged, making both non-identifiable.

\begin{enumerate}

    \item \textbf{Self-activating gene-expression network.} Two biochemical species $X_1$ and $X_2$ interconverting at known rates, where $X_1$ additionally exhibits unknown self-activated production and both species degrade linearly.
    \begin{align*}
        \dot{X}_1 &= f(X_1) - k_1 X_1 + k_2 X_2 - d_1 X_1, \\
        \dot{X}_2 &= k_1 X_1 - k_2 X_2 - d_2 X_2.
    \end{align*}
    Here $f(\vec{u}) = f(X_1)$ is the unknown self-activation rate, with $\vec{u} = \{X_1\}$, $r(\vec{v};\vec{a}) = -d_1 X_1$ has $\vec{a} = \{d_1\}$ (the unknown degradation rate of $X_1$) acting on $\vec{v} = \{X_1\} \subseteq \vec{u}$, absent from $\dot{X}_2$. The function $s(\vec{X}) = -k_1 X_1 + k_2 X_2$ contains the known interconversion rates $k_1$ and $k_2$---since these are known (not inferred) they are not entangled with $f$, and $d_2$ is the unknown degradation rate of $X_2$.

    \item \textbf{Epidemic model with unknown nonlinear recovery.} Susceptible $S$ and infectious $I$ populations, where infectious individuals recover through both a linear recovery rate and a nonlinear rate capturing e.g.\ hospital-capacity constraints.
    \begin{align*}
        \dot{S} &= -\beta SI, \\
        \dot{I} &= \beta SI - f(I) - \gamma I.
    \end{align*}
    Here $f(\vec{u}) = f(I)$ is the unknown nonlinear recovery rate, with $\vec{u} = \{I\}$. $r(\vec{v};\vec{a}) = -\gamma I$ has $\vec{a} = \{\gamma\}$ (the unknown linear recovery rate) acting on $\vec{v} = \{I\} \subseteq \vec{u}$, absent from $\dot{S}$. The infection term, $s(\vec{X};\vec{b}) = \beta SI$ has $\vec{b} = \{\beta\}$ (the unknown transmission rate parameter).

    \item \textbf{Grazer and vegetation dynamics.} A grazer population $G$ consuming vegetation $V$, growing at an unknown rate depending on both $G$ and $V$. $G$ is also subject to constant immigration and linear mortality, while $V$ follows logistic growth.
    \begin{align*}
        \dot{G} &= f(G,V) + \alpha - mG, \\
        \dot{V} &= qV\!\left(1 - \frac{V}{K}\right).
    \end{align*}
    Here $f(\vec{u}) = f(G,V)$ is the unknown growth law, with $\vec{u} = \{G,V\}$. $r(\vec{v};\vec{a}) = \alpha - mG$ has $\vec{a} = \{\alpha,m\}$ (the unknown immigration rate and mortality rate) acting on $\vec{v} = \{G\} \subseteq \vec{u}$, absent from $\dot{V}$; $s(\vec{X};\vec{b}) \equiv 0$ is trivial. The parameters $q$ and $K$ are the vegetation's unknown intrinsic growth rate and carrying capacity.
    
\end{enumerate}


\subsection{Proposition~\ref{prop:complete_UDE_nonident}: fully augmented ODEs} \label{si:examples_complete}

We consider the class of models where every equation has the form $\dot{X}_i = f_i(\vec{X}) + g_i(\vec{X};\vec{p}_i)$ with at least one unknown parameter per equation.

\begin{enumerate}
    \item \textbf{Coupled thermal compartments with unknown exchange laws.} Temperatures $T_1$ and $T_2$ in two coupled compartments, with unknown heat-exchange corrections and explicit linear heat loss to an ambient environment.
    \begin{align*}
        \dot{T}_1 &= f(T_1,T_2) - h_1(T_1 - T_{\mathrm{env}}), \\
        \dot{T}_2 &= g(T_1,T_2) - h_2(T_2 - T_{\mathrm{env}}).
    \end{align*}
    Here $f(T_1,T_2)$ and $g(T_1,T_2)$ are the unknown exchange laws. The known terms contain unknown heat-loss coefficients $h_1$ and $h_2$, which can be absorbed into the corresponding unknown functions.

    \item \textbf{Two competing firms with unknown market effects.} Revenue levels $x_1$ and $x_2$ of two firms operating in partially overlapping but distinct market segments, so that each firm's growth is limited both by its own market saturation and by competitive pressure from the other. Unknown functions $f$ and $g$ capture additional market dynamics not represented by the explicit competition structure---for instance, brand-loyalty effects or differential responses to external demand shocks that depend on both firms' current positions.
    \begin{align*}
        \dot{x}_1 &= f(x_1,x_2) + r_1 x_1\!\left(1 - \frac{x_1 + \alpha_{12}\,x_2}{K_1}\right), \\
        \dot{x}_2 &= g(x_1,x_2) + r_2 x_2\!\left(1 - \frac{x_2 + \alpha_{21}\,x_1}{K_2}\right).
    \end{align*}
    Here $f(x_1,x_2)$ and $g(x_1,x_2)$ are the unknown market-effect laws: $g_1(\vec{X};r_1,K_1,\alpha_{12}) = r_1 x_1(1-(x_1+\alpha_{12}x_2)/K_1)$ contains unknown intrinsic growth rate $r_1$, market capacity $K_1$, and cross-competition coefficient $\alpha_{12}$ (with $\alpha_{12} \neq \alpha_{21}$ reflecting the asymmetry of the two segments), while $g_2(\vec{X};r_2,K_2,\alpha_{21}) = r_2 x_2(1-(x_2+\alpha_{21}x_1)/K_2)$ contains the corresponding parameters $r_2$, $K_2$, and $\alpha_{21}$.

    \item \textbf{Soil moisture and vegetation water dynamics.} Soil moisture $M$ and vegetation water status $V$ (e.g.\ plant water content), coupled through unknown soil-water and plant-water dynamics laws, with a constant water input and linear loss terms.
    \begin{align*}
        \dot{M} &= f(M,V) + P - \ell M, \\
        \dot{V} &= g(M,V) - mV.
    \end{align*}
    Here $f(M,V)$ is the unknown soil--water dynamics law, capturing the effect of soil and vegetation water content on soil--water retention, and $g(M,V)$ is the unknown plant--water dynamics law, capturing the effect of soil moisture on vegetation water uptake. The function $g_1(\vec{X};P,\ell) = P - \ell M$ contains unknown constant water-input rate $P$ (e.g.\ average rainfall) and unknown background moisture-loss rate $\ell$, while $g_2(\vec{X};m) = -mV$ contains unknown vegetation-water loss rate $m$.
\end{enumerate}


\section{Additional examples of local functional identifiability} \label{si:local_identifiability_examples}

In Section~\ref{sec:local_global} we introduced the notion of local functional identifiability, and gave a simple example in Proposition~\ref{prop:local_global}. Here, we give additional examples of different forms of local structural identifiability that can appear in the context of functions. The first two examples are translations of cases of local structural parametric identifiability, while the other two only appear in the context of functions. These examples are not intended to be exhaustive, but rather to illustrate how a wide variety of local functional identifiability scenarios are possible.

\begin{enumerate}
    \item \textbf{Sign ambiguity.} Consider
    \begin{equation*}
        \dot{X} = f(X)^2,
    \end{equation*}
    where $X$ is observed. Here, $f$ is locally, but not globally, identifiable.
    \begin{proof}[Proof]
        If two admissible functions $f_1$ and $f_2$ give the same observed dynamics, then
        \begin{equation*}
            f_1(x)^2 = f_2(x)^2.
        \end{equation*}
        Here, both the alternatives $f_1=f_2$ and $f_1=-f_2$ may hold, implying that at least two potential functions will be consistent with any given observed dynamics. This implies local functional identifiability.
    \end{proof}

    We note that $f$ is globally identifiable at any point where $f_1(x^*)=f_2(x^*)=0$. Furthermore, even if $f$ is assumed to be continuous, whichever of the relations $f_1=f_2$ and $f_1=-f_2$ holds may switch at points where $f_1(x^*)=f_2(x^*)=0$, permitting the generation of a potentially large equivalence class of functions. The generation of such an equivalence class is illustrated in Figure~\ref{fig:local_ident_switching}.

    \item \textbf{Function-label permutation ambiguity.} Consider the two-state model
    \begin{align*}
        \dot{X}_1 &= f(X_1)+g(X_1),\\
        \dot{X}_2 &= f(X_1)g(X_1),
    \end{align*}
    where both $X_1$ and $X_2$ are observed. Here, $f$ and $g$ are locally, but not globally, identifiable.
    \begin{proof}[Proof]
        Let $(f_1,g_1)$ and $(f_2,g_2)$ be two admissible pairs producing the same observed dynamics. Since both state variables are observed, the model equations give
        \begin{align*}
            f_2(x)+g_2(x)=f_1(x)+g_1(x),
            \qquad
            f_2(x)g_2(x)=f_1(x)g_1(x),
        \end{align*}
        on the observed range of $X_1$. From the first relation, $g_2(x)=f_1(x)+g_1(x)-f_2(x)$. Substituting this into the second relation gives
        \begin{equation*}
            \bigl(f_2(x)-f_1(x)\bigr)\bigl(f_2(x)-g_1(x)\bigr)=0.
        \end{equation*}
        Here, both the alternatives $(f_1,g_1)=(f_2,g_2)$ and $(f_1,g_1)=(g_2,f_2)$ may hold, implying that at least two possible pairs of functions will be consistent with any given observed dynamics. This implies local functional identifiability.
    \end{proof}

    We note that $(f,g)$ is globally identifiable at any point where $f_1(x^*)=g_1(x^*)$. Furthermore, even if $f$ and $g$ are assumed to be continuous, whichever of the relations $(f_1,g_1)=(f_2,g_2)$ and $(f_1,g_1)=(g_2,f_2)$ holds may switch at points where $f_1(x^*)=g_1(x^*)$, permitting the generation of a potentially large equivalence class of function pairs.

    \item \textbf{Argument-reflection ambiguity.} Consider
    \begin{align*}
        \dot{X}_1 &= f(X_1)+f(-X_1),\\
        \dot{X}_2 &= f(X_1)f(-X_1),
    \end{align*}
    where both states are observed. Here, $f$ is locally, but not globally, identifiable.
    \begin{proof}[Proof]
        Let $f_1$ and $f_2$ give the same observed dynamics. Then
        \begin{align*}
            f_2(x)+f_2(-x)=f_1(x)+f_1(-x),
            \qquad
            f_2(x)f_2(-x)=f_1(x)f_1(-x),
        \end{align*}
        for the observed values $x$ of $X_1$. From the first relation, $f_2(-x)=f_1(x)+f_1(-x)-f_2(x)$. Substituting this into the second relation gives
        \begin{equation*}
            \bigl(f_2(x)-f_1(x)\bigr)\bigl(f_2(x)-f_1(-x)\bigr)=0.
        \end{equation*}
        Here, both the alternatives $f_1=f_2$ and $f_2(x)=f_1(-x)$ may hold, implying that at least two possible functions will be consistent with any given observed dynamics. This implies local functional identifiability.
    \end{proof}

    We note that $f$ is globally identifiable at any point where $f_1(x^*)=f_1(-x^*)$. Furthermore, even if $f$ is assumed to be continuous, whichever of the relations $f_1(x^*)=f_2(x^*)$ and $f_2(x^*)=f_1(-x^*)$ holds may switch at points where $f_1(x^*)=f_1(-x^*)$, permitting the generation of a potentially large equivalence class of functions.

    \item \textbf{Input-coordinate ambiguity.} Consider
    \begin{align*}
        \dot{X}_1 &= f(X_1,X_2)+f(X_2,X_1),\\
        \dot{X}_2 &= f(X_1,X_2)f(X_2,X_1),
    \end{align*}
    where both states are observed. Here, $f$ is locally, but not globally, identifiable.
    \begin{proof}[Proof]
        Let $f_1$ and $f_2$ give the same observed dynamics. Since both states are observed,
        \begin{align*}
            f_2(x,y)+f_2(y,x)&=f_1(x,y)+f_1(y,x), \\
            f_2(x,y)f_2(y,x)&=f_1(x,y)f_1(y,x)
        \end{align*}
        on the observed pairs $(x,y)=(X_1,X_2)$. From the first relation, $f_2(y,x)=f_1(x,y)+f_1(y,x)-f_2(x,y)$. Substituting this into the second relation gives
        \begin{equation*}
            \bigl(f_2(x,y)-f_1(x,y)\bigr)\bigl(f_2(x,y)-f_1(y,x)\bigr)=0.
        \end{equation*}
        Here, both the alternatives $f_1=f_2$ and $f_2(x,y)=f_1(y,x)$ may hold, implying that at least two possible functions will be consistent with any given observed dynamics. This implies local functional identifiability.
    \end{proof}

    We note that $f$ is globally identifiable at any point where $f_1(x^*,y^*)=f_1(y^*,x^*)$. Furthermore, even if $f$ is assumed to be continuous, whichever of the relations $f_1(x,y)=f_2(x,y)$ and $f_2(x,y)=f_1(y,x)$ holds may switch at points where $f_1(x^*,y^*)=f_1(y^*,x^*)$, permitting the generation of a potentially large equivalence class of functions.
\end{enumerate}

Many of these cases exhibit a \emph{switching} behaviour, which generates a large family of alternative functional forms. Here, the function's value is pinned down at a set of individual points,
\begin{equation*}
    f(x_1), f(x_2), \dots, f(x_n),
\end{equation*}
while on each interval $(x_i,x_{i+1})$ two (or more) valid functional forms exist. The full set of admissible functions is then given by every combination of these forms across the intervals. An example is shown in Figure~\ref{fig:local_ident_switching}.

\begin{figure}
    \centering
    \includegraphics[width=0.90\linewidth]{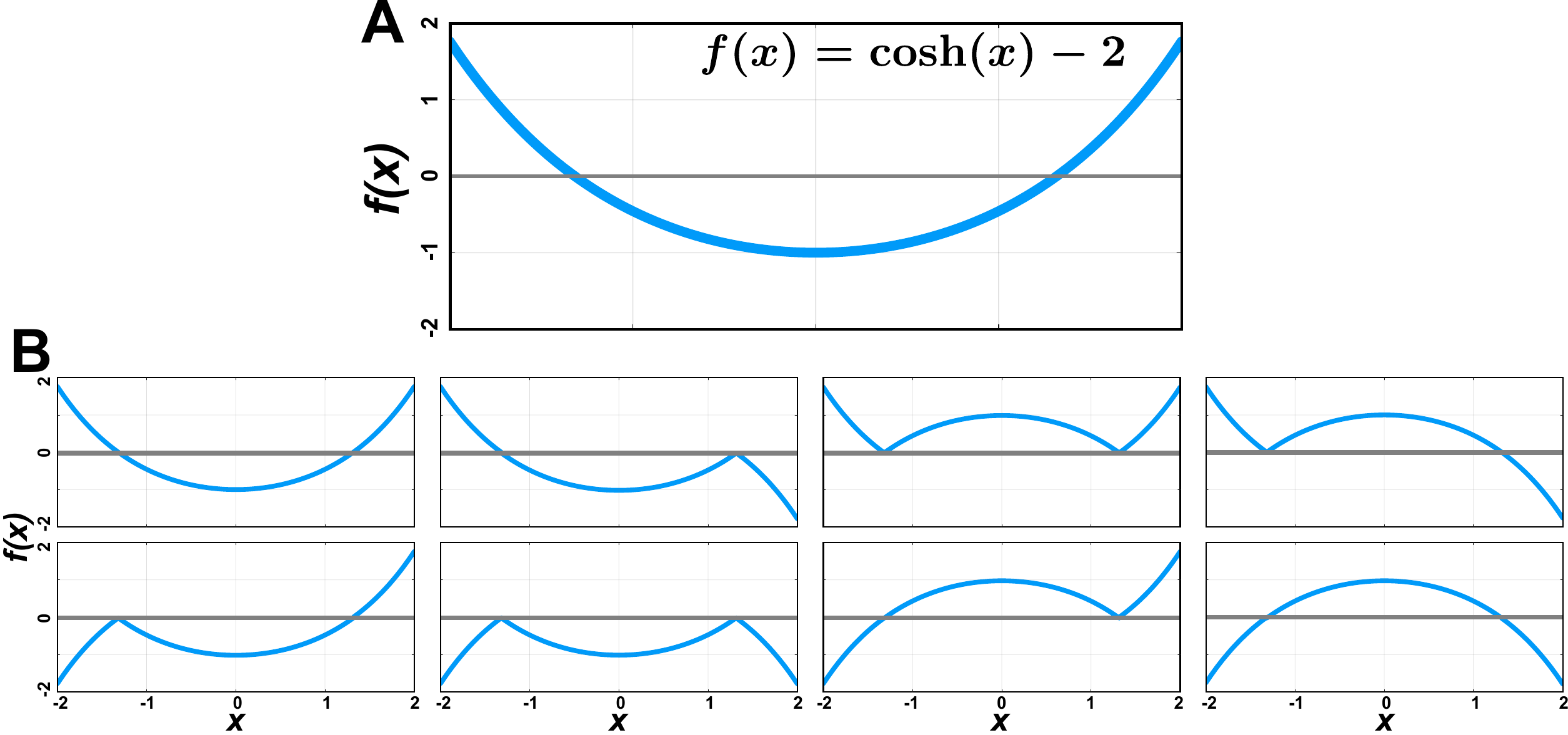}
    \caption{\textbf{Sign switching generates many locally identifiable alternative functions.} Consider the sign ambiguity model, $\dot{X} = f(X)^2$, where $f$ is an unknown function to be recovered from data. Here, $f$ is (as proven above) locally structurally identifiable: for any given $x^*$, $f$'s value can only be determined up to the equivalence class $\{\hat{f}(x^*), -\hat{f}(x^*)\}$, for some $\hat{f}$. If we further assume that $f$ is continuous, a switching phenomenon arises whereby the sign may change at any point where $\hat{f}(x) = 0$. Hence, if the data are generated by a function with $n$ zeros, $f$ is (locally) structurally identifiable up to $2^{n+1}$ alternative forms. As an example, assume that the data are generated by the model $\dot{X} = f(X)^2$ with $f(x) = \cosh(x) - 2$ (depicted in A). The eight potential functional forms depicted in B then all yield identical data, and $f$ can only be identified up to these alternatives. If $f$ is additionally assumed to be smooth, only the two forms in the top left and bottom right corners remain (corresponding to $f(x)$ and $-f(x)$).}
    \label{fig:local_ident_switching}
\end{figure}


\section{Recovery of non-invertible functions} \label{si:noninvertible_recovery}

In Section~\ref{sec:diffalg_one_parameter}, we noted that the identifiability criteria hold even for functions that are not globally invertible, provided that they are locally invertible. Here, we provide a numerical example where one such function is correctly recovered from synthetic data.

We consider the model described by Equations~\eqref{eq:diffalg_model1_dXdt}--\eqref{eq:diffalg_model1_dYdt} for the following values of the unknown components: $d = 0.4$ and $f(x) = 0.1 + 1.5 x^3 / (x^3 + 3^3) - 0.1x$. Notably, $f$ is not globally invertible. From this model instance, using four different initial conditions (Table~\ref{table:noninvertible_recovery_data}), we generate synthetic datasets according to the method detailed in Section~\ref{si:method_numerics_data}. To these data, we then fit the model's unknown components using the approach detailed in Section~\ref{si:method_numerics_fitting} (Figure~\ref{fig:noninvertible_recovery}A). Adapting the ensemble plot approach used to assess practical functional identifiability~\cite{loman_functional_2025}, we confirm that the correct function is uniquely recovered, indicating structural identifiability of the full functional form (Figure~\ref{fig:noninvertible_recovery}B).

\begin{figure}
    \centering
    \includegraphics[width=0.90\linewidth]{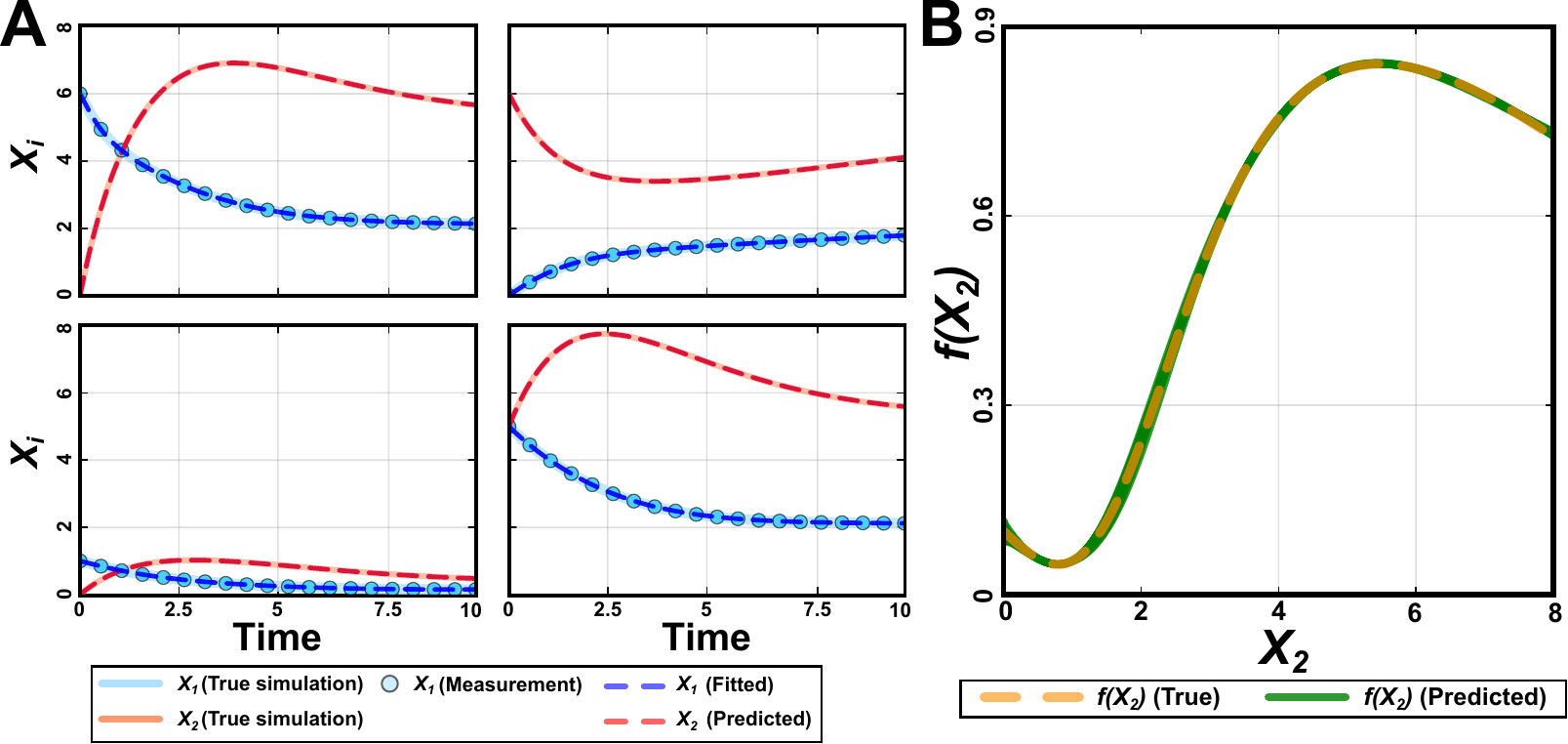}
    \caption{\textbf{Non-invertible functions can be uniquely recovered from data.} (A) Four initial conditions (Table~\ref{table:noninvertible_recovery_data}) of the model described by Equations~\eqref{eq:diffalg_model1_dXdt}--\eqref{eq:diffalg_model1_dYdt} with $d = 0.4$ and $f(x) = 0.1 + 1.5 x^3 / (x^3 + 3^3) - 0.1x$. From these, ground-truth simulations (faint lines) are generated, from which synthetic data (dots) are sampled as described in Section~\ref{si:method_numerics_data}. To these data, the model is fitted (dashed lines). The fitted model simulations, both for the observed ($X_1$) and the unobserved ($X_2$) variable, align well with the ground-truth simulations. (B) The fitting process is repeated, and, for the $25$ best fitting results, the recovered unknown function $f$ is plotted (green lines). The ensemble of plotted functions all align with the ground-truth function (dashed yellow lines), indicating that the unknown function $f$ is identifiable and can be uniquely recovered.}
    \label{fig:noninvertible_recovery}
\end{figure}


\section{Numerical methodology} \label{si:method_numerics}

\subsection{Synthetic data generation} \label{si:method_numerics_data}

In Figure~\ref{fig:mutual_activation_numerics} and Figure~\ref{fig:noninvertible_recovery}, synthetic data are generated from ground-truth simulations. In each case, and for each initial condition, we solve the model's corresponding differential equation on a time interval $[0,t_{\mathrm{end}}]$. We then generate the data as $n+1$ evenly spaced samples of the observed variable ($X_{\mathrm{obs}}$), without any noise added:
\begin{equation*}
    X_{\mathrm{obs}}(0), \quad X_{\mathrm{obs}}(t_{\mathrm{end}}/n), \quad X_{\mathrm{obs}}(2t_{\mathrm{end}}/n), \quad \dots, \quad X_{\mathrm{obs}}(t_{\mathrm{end}}).
\end{equation*}
The initial conditions, time intervals, and numbers of samples are detailed in Tables~\ref{table:mutual_activation_numerics_data} and~\ref{table:noninvertible_recovery_data}.

\begin{table}[!htbp]
    \centering
    \begin{tabular}{|c||c|}
    \hline
    Property & Value \\ 
    \hline
    \hline
    $[0,t_{\mathrm{end}}]$ &$[0,20.0]$ \\ 
    \hline
    $n$ & $25$ \\ 
    \hline
    $(X_{1,1}(0), X_{2,1}(0))$ & $(1.8,0.1)$ \\ 
    \hline
    $(X_{1,2}(0), X_{2,2}(0))$ & $(0.1,6.0)$ \\ 
    \hline
    $(X_{1,3}(0), X_{2,3}(0))$ & $(2.5,0.1)$ \\ 
    \hline
    $(X_{1,4}(0), X_{2,4}(0))$ & $(2.0,5.0)$ \\ 
    \hline
    $(X_{1,5}(0), X_{2,5}(0))$ & $(0.01,0.01)$ \\ 
    \hline
    $(X_{1,6}(0), X_{2,6}(0))$ & $(6.0,1.0)$ \\ 
    \hline
    $(X_{1,7}(0), X_{2,7}(0))$ & $(0.8,0.01)$ \\ 
    \hline
    $(X_{1,8}(0), X_{2,8}(0))$ & $(1.4,0.01)$ \\ 
    \hline
    \end{tabular}
    \medskip
    \caption{\textbf{Synthetic data settings for Figure~\ref{fig:mutual_activation_numerics}.}} 
    \label{table:mutual_activation_numerics_data}
\end{table}

\begin{table}[!htbp]
    \centering
    \begin{tabular}{|c||c|}
    \hline
    Property & Value \\ 
    \hline
    \hline
    $[0,t_{\mathrm{end}}]$ &$[0,10.0]$ \\ 
    \hline
    $n$ & $20$ \\ 
    \hline
    $(X_{1,1}(0), X_{2,1}(0))$ & $(6.0,0.0)$ \\ 
    \hline
    $(X_{1,2}(0), X_{2,2}(0))$ & $(0.0,6.0)$ \\ 
    \hline
    $(X_{1,3}(0), X_{2,3}(0))$ & $(1.0,0.0)$ \\ 
    \hline
    $(X_{1,4}(0), X_{2,4}(0))$ & $(5.0,5.0)$ \\ 
    \hline
    \end{tabular}
    \medskip
    \caption{\textbf{Synthetic data settings for Figure~\ref{fig:noninvertible_recovery}.}} 
    \label{table:noninvertible_recovery_data}
\end{table}

\subsection{Differential equation fitting} \label{si:method_numerics_fitting}

The differential equations are fitted to data by adapting the method used in~\cite{loman_functional_2025}, where the \emph{universal differential equation} approach is used to recover unknown functions from data~\cite{rackauckas_universal_2021}. Here, each unknown function is replaced with a neural network, which is itself parameterised by a large number of scalar parameters. If the neural network is made arbitrarily large (i.e.\ depends on arbitrarily many parameters), it should be able to approximate any function (including the unknown function's true form) arbitrarily well. This reduces the problem to a standard parameter fitting problem, in which the model's (scalar) mechanistic parameters and the neural network's (scalar) parameters are fitted to the data simultaneously. When running the fitting procedure, we ensure that the initial condition values of the unobserved variable are also fitted, i.e. are considered unknown.

In practice, we minimise a sum-of-squares loss function, i.e.~the sum of the squared differences between the model simulations and the data. We first run the Adam optimiser for $50\,000$ steps, and then follow up with L-BFGS to pinpoint the nearest local minimum accurately. For the neural networks, we use small feed-forward architectures (with fewer than 20 internal nodes) and the softplus activation function. Additional details can be found in~\cite{loman_functional_2025}, and the code is linked in the Code availability section.

We used the Lux.jl package for neural network implementation, Catalyst.jl for differential equation implementation, and ModelingToolkitNeuralNets.jl for UDE creation \cite{pal_lux_2022,loman_catalyst_2023,noauthor_scimlmodelingtoolkitneuralnetsjl_2025}. Models were fitted using the Optimization.jl (interface and Adam) and Optim.jl (L-BFGS implementation) packages \cite{mogensen_optim_2018,dixit2023optimization}. ODEs were simulated using OrdinaryDiffEq.jl \cite{rackauckas_differentialequationsjl_2017}. 


\section{Chemical reaction network model derivation} \label{si:crn_derivation}

A chemical reaction network model consists of a set of species and a set of reactions. Here we consider a model with species $X_1$ and $X_2$ and the following reaction events:
\begin{align*}
\ce{ $X_1$ &->C[$1$] $X_1 + X_2$, } \\
\ce{ $X_2$ &->C[$g(X_2)$] $X_1$,} \\
\ce{ $X_1$ &->C[$d$] \varnothing,}
\end{align*}
where $d>0$ is a parameter and $g$ is a function of the concentration of $X_2$. Application of the law of mass action gives the following ODE model:
\begin{align*}
    \dot{X}_1 &= g(X_2)X_2 - dX_1, \\
    \dot{X}_2 &= X_1 - g(X_2)X_2.
\end{align*}
Setting $f(X_2) = g(X_2)X_2$ generates the model in Equations~\eqref{eq:crn_dXdt}--\eqref{eq:crn_dYdt}.

\end{document}